\documentclass[11pt]{article}
\usepackage{amsmath, amscd, amssymb, latexsym, epsfig, color, amsthm}
\numberwithin{equation}{section}

\newtheorem{theorem}{Theorem}[section]
\newtheorem{proposition}[theorem]{Proposition}

\newtheorem{corollary}[theorem]{Corollary}
\newtheorem{conjecture}[theorem]{Conjecture}
\newtheorem{lemma}[theorem]{Lemma}

\theoremstyle{definition}
\newtheorem{definition}[theorem]{Definition}
\newtheorem{example}[theorem]{Example}
\newtheorem{remark}[theorem]{Remark}
\newtheorem{notation}[theorem]{Notation}
\DeclareMathOperator{\skel}{Skel}
\DeclareMathOperator\lk{\mathrm{lk}}

\DeclareMathOperator{\dist}{\mathrm{dist}}

\newcommand{\BH}{\mathcal{BH}}
\newcommand{\BM}{BM}
\newcommand{\ST}{\mathcal{ST}}
\newcommand{\field}{{\bf k}}

\newcommand{\R}{{\mathbb R}}
\newcommand{\Q}{{\mathbb Q}}
\newcommand{\Z}{{\mathbb Z}}

\newcommand{\C}{{\mathcal C}}
\newcommand{\X}{{\mathcal X}}
\newcommand{\halffloor}{\lfloor \frac{d}{2}\rfloor}

\title{Lower Bound Theorems and a Generalized Lower Bound Conjecture for balanced simplicial complexes}
\author{Steven Klee\\
\small Department of Mathematics \\[-0.8ex]
\small Seattle University\\[-0.8ex]
\small Seattle, WA 98122, USA\\[-0.8ex]
\small \texttt{klees@seattleu.edu}
\and Isabella Novik
\thanks{Research is partially
supported by NSF grants DMS-1069298 and DMS-1361423}\\
\small Department of Mathematics\\[-0.8ex]
\small University of Washington\\[-0.8ex]
\small Seattle, WA 98195-4350, USA\\[-0.8ex]
\small \texttt{novik@math.washington.edu}
}

\begin{document}

\maketitle


\begin{abstract}
A  $(d-1)$-dimensional simplicial complex is called balanced if its underlying graph admits a proper $d$-coloring. We show that many well-known face enumeration results have natural balanced analogs (or at least conjectural analogs). Specifically, we prove the balanced analog of the celebrated Lower Bound Theorem for normal pseudomanifolds and characterize the case of equality; we introduce and characterize the balanced analog of the Walkup class; we propose the balanced analog of the Generalized Lower Bound Conjecture and establish some related results. We close with constructions of balanced manifolds with few vertices. 
\end{abstract}

\section{Introduction}

This paper is devoted to the study of face numbers of balanced simplicial complexes. A $(d-1)$-dimensional simplicial complex $\Delta$ is called \textit{balanced} if the graph of $\Delta$ is $d$-colorable. In other words, each vertex of $\Delta$ can be assigned one of  $d$ colors in such a way that no edge has both endpoints of the same color. This class of complexes was introduced by Stanley  \cite{Stanley-balanced} where he called them \textit{completely balanced} complexes.
 
Balanced complexes form a fascinating class of objects that arise often in combinatorics, algebra, and topology.  For instance, the barycentric subdivision of \textit{any} regular CW complex is balanced; therefore, every triangulable space has a balanced triangulation.  A great deal of research has been done on the flag face numbers of balanced spheres that arise as the barycentric subdivision of regular CW-complexes in the context of the $cd$-index, see \cite{Stanley-cd, Karu-cd, Ehrenborg-Karu-cd} and many references mentioned therein.  In contrast, we know very little about the face numbers of an arbitrary balanced simplicial sphere or manifold: for instance, even the answer to the question of ``What is the smallest number of vertices that a balanced triangulation of a closed $(d-1)$-manifold that is not a sphere can have?" seems to be unknown. One of our results provides such an answer for the categories of PL manifolds and homology manifolds.  
 
Stanley \cite{Stanley-balanced} initiated the study of (flag) face numbers of balanced Cohen--Macaulay complexes. His paper together with the work of Bj\"orner--Frankl--Stanley \cite{Bjorner-Frankl-Stanley} provided a complete characterization of all possible flag $f$-numbers of such complexes.  On the level of face numbers, this characterization  says that an integer vector is the $h$-vector of a $(d-1)$-dimensional balanced Cohen--Macaulay complex if and only if it is the $f$-vector of a $d$-colorable simplicial complex. Hence the $h$-vectors of balanced Cohen--Macaulay complexes satisfy Kruskal--Katona type inequalities of \cite{Frankl-Furedi-Kalai}. On the other hand, in the class of \textit{all simplicial polytopes}, the celebrated $g$-theorem \cite{Stanley-gthm, Billera-Lee} provides much stronger restrictions, and the $g$-conjecture posits that the same restrictions hold for all triangulated  spheres. Moreover, a part of the $g$-theorem --- the Lower Bound Theorem --- holds even in the generality of normal pseudomanifolds \cite{Barnette-LBT-pseudomanifolds, Kalai-rigidity, Fogelsanger, Tay}.
 
Are there balanced analogs of the aforementioned results? This question served as the main motivation and starting point for this paper.  Another motivation came from a flurry of recent activity on finding vertex-minimal triangulations of various manifolds (see, for instance, \cite{Lutz-thesis, Lutz-Sulanke-Swartz, Bagchi-Datta-bundles, Chestnut-Sapir-Swartz} as well the Manifold Page \cite{Lutz-page}) and our desire to  find balanced analogs of at least some of these constructions.
 
To our surprise, we found that several classic face enumeration results have natural balanced analogs or at least conjectural balanced analogs. Our results can be summarized as follows. We defer all the definitions until later sections. Many basic definitions and results pertaining to  simplicial complexes and balanced simplicial complexes are collected in Section 2. 
 
\begin{itemize}
\item In Section 3, we show that the balanced analog of the Lower Bound Theorem (established in \cite{Goff-Klee-Novik} for balanced spheres and in \cite{Browder-Klee} for balanced manifolds) continues to hold for all balanced triangulations of normal pseudomanifolds, see Theorem \ref{balancedLBT}.
 
\item In Section 4, we treat extremal cases of the balanced Lower Bound Theorem: in analogy with the non-balanced situation  \cite{Kalai-rigidity}, we show that a balanced normal pseudomanifold of dimension $d-1\geq 3$ satisfies the balanced Lower Bound Theorem with equality if and only if it is a ``stacked cross-polytopal sphere," see Theorem \ref{balancedLBT-equality}.
 
\item In Section 4, we also introduce the balanced analog of Walkup's class  and  show that in complete analogy with the non-balanced case \cite{Walkup, Kalai-rigidity}, for $d\geq 5$, a balanced  $(d-1)$-dimensional complex is in the balanced  Walkup class if and only if all its vertex links are stacked cross-polytopal spheres, see Corollary \ref{balanced8.4}.
 
\item In Section 5, we posit the balanced analog of McMullen-Walkup's Generalized Lower Bound Conjecture \cite{McMullen-Walkup} (together with the treatment of equality cases), see Conjecture \ref{balancedGLBC}. We should stress that while the Generalized Lower Bound Conjecture for the class of simplicial polytopes is now a theorem (the inequalities of the GLBC form a part of the $g$-theorem, and the treatment of equality cases was recently completed in \cite{Murai-Nevo}),  the balanced GLBC is at present a conjecture, even for the class of balanced polytopes. We verify the ``easy" part of the equality case of this conjecture, see Theorem \ref{thm:balancedGLBC-2b-implies-2a},  and prove that, in analogy with the results from \cite{Bagchi-Datta-stellated, Murai-Nevo2}, manifolds with the balanced $r$-stacked property have a unique $r$--stacked cross-polytopal decomposition, see Theorem \ref{uniqueness_thm}.
 
\item In Section 6 we discuss balanced triangulations of manifolds. In particular, we construct  a vertex minimal balanced triangulation of the  (orientable or non-orientable, depending on the parity of $d$) $\mathbb{S}^{d-2}$-bundle over $\mathbb{S}^1$, see Theorem \ref{thm:3d-vertex-construction}. Our construction can be considered a balanced analog of K\"uhnel's construction \cite{Kuhnel}. By introducing only two more vertices, we are able to construct balanced triangulations of both of these bundles, see Theorem \ref{Delta_k,d}. 
\end{itemize}
 
Although several of our proofs follow along the lines of their non-balanced analogs, we believe that these results are rather unexpected and provide new insights in the theory of the face numbers for balanced complexes. We also hope that this paper will motivate an even more thorough study of balanced complexes (note that at present we do not have even a conjectural balanced analog of the Upper Bound Theorem) as well as will lead to new construction techniques for balanced manifolds.

\section{Preliminaries}

\subsection{Topological and combinatorial invariants of simplicial complexes}
A \textit{simplicial complex} $\Delta$ on a finite vertex set $V = V(\Delta)$ is a collection of subsets of $V$ called \textit{faces} with the property that (i) $\{v\}\in\Delta$ for all $v\in V$, and (ii) if $F \in \Delta$ and $G \subseteq F$, then $G \in \Delta$.  The \textit{dimension} of a face $F \in \Delta$ is $\dim(F) = |F|-1$, and the dimension of $\Delta$ is $\dim(\Delta) = \max\{\dim(F)\ : \ F \in \Delta\}$. For brevity, we refer to an $i$-dimensional face as an $i$-face. A \textit{facet} in $\Delta$ is a maximal face under inclusion, and we say that $\Delta$ is \textit{pure} if all of its facets have the same dimension.  

The \textit{link} of a face $F \in \Delta$ is the subcomplex $$\lk_{\Delta}(F) = \{G \in \Delta\ : \ F \cap G = \emptyset \text{ and } F \cup G \in \Delta\}.$$  Intuitively, the link encodes the local structure of the simplicial complex $\Delta$ around the face $F$.  If $W \subseteq V(\Delta)$ is any subset of vertices, we define the \textit{restriction} of $\Delta$ to $W$ to be the subcomplex $$\Delta[W] = \{F \in \Delta\ : \ F \subseteq W\}.$$  The \textit{$i$-skeleton} of $\Delta$, $\skel_i(\Delta)$, is the subcomplex of all faces of $\Delta$ of dimension at most $i$. The 1-skeleton is also called the \textit{graph} of $\Delta$.

Although the above definition of a simplicial complex is an abstract combinatorial construction, to any (abstract) simplicial complex $\Delta$ there is an associated topological space $\|\Delta\|$ called the \textit{geometric realization} of $\Delta$ (or the underlying topological space), which contains a $(k-1)$-dimensional geometric simplex for each $(k-1)$-face of $\Delta$.  We frequently will not emphasize the distinction between $\Delta$ and $\|\Delta\|$ and will refer to topological properties of $\|\Delta\|$ simply as topological properties of $\Delta$.  

We will be interested in studying certain relaxations of the family of triangulations of spheres and manifolds.  A $(d-1)$-dimensional simplicial complex $\Delta$ is a simplicial $(d-1)$-sphere (respectively simplicial ball or simplicial manifold) if its geometric realization $\|\Delta\|$ is homeomorphic to a sphere (resp.~ball or manifold) of dimension $d-1$.  

It follows from the excision axiom that if $\Delta$ is a simplicial complex and $p \in \|\Delta\|$ is a point that lies in the relative interior of a face $F$ of $\Delta$, then $H_i(\|\Delta\|,\|\Delta\|-p;\field) \cong \widetilde{H}_{i-|F|}(\lk_{\Delta}(F);\field)$ (see \cite[Lemma 3.3]{Munkres}).  Here, $\widetilde{H}_*(\Delta;\field)$ denotes the reduced simplicial homology groups of $\Delta$ with coefficients in $\field$ and $H_*(\|\Delta\|,\|\Delta\|-p;\field)$ denotes the relative homology groups. (We will also use $\beta_i(\Delta;\field)$ to denote the reduced Betti numbers of $\Delta$ with coefficients in $\field$: $\beta_i(\Delta;\field):= \dim_{\field}\widetilde{H}_i(\Delta;\field)$.) On the other hand, if $\Delta$ triangulates a manifold without boundary and $p$ is a point of $\|\Delta\|$, then the pair $(\|\Delta\|,\|\Delta\|-p)$ has the relative homology of a $(d-1)$-sphere.

Thus as a relaxation of the family of simplicial spheres/manifolds, we say that $\Delta$ is a homology $(d-1)$-sphere over a field $\field$ (or a $\field$-homology sphere) if $\widetilde{H}_*(\lk_{\Delta}(F);\field) \cong \widetilde{H}_*(\mathbb{S}^{d-|F|-1};\field)$ for every face $F \in \Delta$ (including the empty face), and that  $\Delta$ is a (closed) homology  $(d-1)$-manifold over $\field$ if $\widetilde{H}_*(\lk_{\Delta}(F);\field) \cong \widetilde{H}_*(\mathbb{S}^{d-|F|-1};\field)$ for every nonempty face $F \in \Delta$.  All simplicial spheres (resp.~manifolds) are homology spheres (resp.~manifolds) over any field. The class of homology $2$-spheres coincides with that of simplicial $2$-spheres, and consequently, the class of homology 3-manifolds coincides with that of simplicial 3-manifolds. However, for every $d\geq 4$, there exist homology $(d-1)$-spheres that are not simplicial $(d-1)$-spheres, and for  every $d\geq 5$, there exist homology $(d-1)$-manifolds that are not simplicial $(d-1)$-manifolds.

A $(d-1)$-dimensional simplicial complex $\Delta$ is called a \textit{normal pseudomanifold} (or a normal $(d-1)$-pseudomanifold) if (i) $\Delta$ is connected and pure, (ii) every $(d-2)$-face (or \textit{ridge}) of $\Delta$ is contained in exactly two facets, and (iii) the link of each face of dimension $\leq d-3$ is connected. In particular, every connected (closed) homology manifold is a normal pseudomanifold. For $d=3$, the class of normal $(d-1)$-pseudomanifolds coincides with the class of connected homology $(d-1)$-manifolds, but for $d>3$, the former class is much larger than the latter. It is well-known and easy to check that if $\Delta$ is a normal $(d-1)$-pseudomanifold and $F$ is a face of $\Delta$ of dimension at most $d-2$, then the link of $F$ is also a normal pseudomanifold. Another useful property of normal pseudomanifolds is that their \textit{facet-ridge graphs} (also known as dual graphs) are connected. In other words, if $\Delta$ is a normal $(d-1)$-pseudomanifold and $F$, $G$ are two arbitrary facets of $\Delta$, then there exists a sequence of facets $F=F_0, F_1, \ldots, F_{p-1}, F_p=G$ of $\Delta$ with the property that every pair of consecutive elements of this sequence share a $(d-2)$-face (see, for instance, \cite[Lemma 2.1]{BagchiDatta-LBT}).

Other topological notions that will be important for this paper are the constructions of join, connected sum, and handle addition.  
Let $\Gamma$ and $\Delta$ be simplicial complexes on disjoint vertex sets.  The \textit{join} of $\Gamma$ and $\Delta$, denoted $\Gamma*\Delta$, is the simplicial complex on vertex set $V(\Gamma) \cup V(\Delta)$ whose faces are $\{\sigma \cup \tau\ : \ \sigma \in \Gamma \text{ and } \tau \in \Delta\}$.  

Let $\Gamma$ and $\Delta$ be pure simplicial complexes of the same dimension on disjoint vertex sets.  Let $F$ and $G$ be facets of $\Gamma$ and $\Delta$ respectively, and let $\varphi: F \rightarrow G$ be a bijection between the vertices of $F$ and the vertices of $G$.  The \textit{connected sum} of $\Gamma$ and $\Delta$, denoted $\Gamma \#_{\varphi} \Delta$ or simply  $\Gamma \# \Delta$, is the simplicial complex obtained by identifying the vertices of $F$ and $G$ (and all faces on those vertices) according to the bijection $\varphi$ and removing the facet corresponding to $F$ (which has been identified with $G$).  This coincides with the familiar topological construction where we have removed the relative interiors of the facets $F$ and $G$, which are open balls in $\Gamma$ and $\Delta$ respectively, and glued them together along their boundaries.  

Finally, let $\Delta$ be a pure simplicial complex of dimension $d-1$, and let $F$ and $F'$ be facets of $\Delta$ with disjoint vertex sets.  If there is a bijection $\varphi: F \rightarrow F'$ such that $v$ and $\varphi(v)$ do not have a common neighbor in $\Delta$ for every $v \in F$, the simplicial complex $\Delta^\varphi$ obtained from $\Delta$ by identifying the vertices of $F$ and $F'$ (and all faces on those vertices) and removing the facet corresponding to $F$ (which has been identified with $F'$) is called a \textit{handle addition} to $\Delta$.  The requirement that $v$ and $\varphi(v)$ do not have a common neighbor in $\Delta$ ensures that $\Delta^\varphi$ is a simplicial complex.  

The most natural combinatorial invariants of a simplicial complex are its \textit{$f$-numbers}.  If $\Delta$ is a $(d-1)$-dimensional simplicial complex, define $f_{i}(\Delta)$ to count the number of $i$-faces in $\Delta$ for any $-1 \leq i \leq d-1$.  As long as $\Delta$ is nonempty, we have $f_{-1}(\Delta) = 1$, which corresponds to the empty face in $\Delta$.  The $f$-numbers of $\Delta$ are often arranged in a single vector $f(\Delta) := (f_{-1}(\Delta), f_0(\Delta),\ldots,f_{d-1}(\Delta))$ called the \textit{$f$-vector} of $\Delta$.  

Typically it is more convenient to study a certain transformation of the $f$-numbers of a simplicial complex called its \textit{$h$-numbers}.  If $\Delta$ is a $(d-1)$-dimensional simplicial complex, we define the \textit{$h$-vector} of $\Delta$ to be the vector $h(\Delta):=(h_0(\Delta),h_1(\Delta),\ldots,h_d(\Delta))$, whose entries are given by $$h_j(\Delta) := \sum_{i=0}^j (-1)^{j-i}\binom{d-i}{d-j} f_{i-1}(\Delta).$$ The above formula can be inverted to express each $f$-number as a \textit{nonnegative} linear combination of the $h$-numbers, and hence knowing the $f$-numbers of a simplicial complex is equivalent to knowing its $h$-numbers.  Moreover, inequalities on the $h$-numbers of a simplicial complex translate directly into inequalities on the $f$-numbers (while the converse is not true).  

For this reason, it is often preferable to study $h$-numbers in favor of $f$-numbers.  Furthermore, the $h$-numbers arise naturally in the algebraic study of the Stanley-Reisner ring of a simplicial complex.  As we will not use this tool in the paper, we refer to Stanley's book \cite{Stanley-green-book} for further information.  Also, many identities involving $f$-numbers can be stated more cleanly in terms of $h$-numbers.  For example, if $\Delta$ is a homology $(d-1)$-sphere, then the Dehn-Sommerville relations \cite{Klee-DS} state that $h_j(\Delta) = h_{d-j}(\Delta)$ for all $0 \leq j \leq d$.  This means that the entire $h$-vector of a simplicial homology sphere is determined by the values of $h_0, h_1, \ldots, h_{\halffloor}$.   

As a further step, it is useful to consider the successive differences between these numbers, called the \textit{$g$-numbers}: $g_0(\Delta) := 1$ and $g_j(\Delta) := h_j(\Delta)-h_{j-1}(\Delta)$ for $1\leq j\leq d$, and to arrange the first half of these numbers in the \textit{$g$-vector} of $\Delta$,  $g(\Delta) := (g_0(\Delta),g_1(\Delta),\ldots,g_{\halffloor}(\Delta))$. 

The famous $g$-theorem of Stanley \cite{Stanley-gthm} and Billera-Lee \cite{Billera-Lee} completely characterizes the integer vectors that can arise as the $g$-vector of the boundary complex of a simplicial $d$-polytope.  

\begin{theorem}{\rm{(The $g$-theorem \cite{Billera-Lee, Stanley-gthm})}}
An integer vector $\mathbf{g} = (g_0,g_1,\ldots,g_{\halffloor})$ is the $g$-vector of a simplicial $d$-polytope if and only if $g_0 = 1$ and $\mathbf{g}$ is an $M$-vector.
\end{theorem}
\noindent In particular, the $g$-numbers of a simplicial $d$-polytope are nonnegative, and hence the $h$-vector is unimodal.  

\subsection{Balanced simplicial complexes}
In this paper we will be interested in studying the family of balanced simplicial complexes (introduced in \cite{Stanley-balanced}), which are equipped with a vertex coloring that imposes additional combinatorial structure.  

\begin{definition}
Let $\Delta$ be a $(d-1)$-dimensional simplicial complex.  We say that $\Delta$ is \textit{balanced} if the vertices of $\Delta$ can be partitioned as $V(\Delta) = V_1 \amalg V_2 \amalg \cdots \amalg V_d$ in such a way that $|F \cap V_i|\leq 1$ for all faces $F \in \Delta$. (Here, $\amalg$ denotes the disjoint union.) 
\end{definition}

We can view this vertex partition as a coloring of the vertices in which the vertices in the set $V_i$ have received color $i$.  We use the notation $[d] :=\{1,2,\ldots, d\}$ to denote the set of colors.  The condition that each face $F$ satisfies $|F \cap V_i| \leq 1$ says that no face contains two vertices of the same color or equivalently that no two vertices of the same color are connected by an edge.  Thus the vertex partition induces a proper $d$-coloring of the underlying graph of $\Delta$.  Since the graph of a $(d-1)$-simplex is a complete graph on $d$ vertices, at least $d$ colors are required to properly color the vertices of a $(d-1)$-dimensional simplicial complex.  Balanced complexes are those for which such a minimal coloring is possible.  In the literature, the family of complexes defined above is sometimes called the family of \textit{completely balanced} simplicial complexes.

The canonical example of a balanced simplicial complex is the order complex (also known as the barycentric subdivision) of any regular CW complex.  If $K$ is a regular CW complex, the \textit{order complex  $\Delta(K)$ of $K$} is the simplicial complex whose vertices are the nonempty faces of $K$ and whose faces correspond to chains of nonempty faces $\tau_0 \subsetneq \tau_1 \subsetneq \cdots \subsetneq \tau_k$ in $K$. Letting $V_i(\Delta(K))$ consist of $(i-1)$-faces of $K$ then makes $\Delta(K)$ into a balanced complex as no chain contains two faces of the same dimension. (In fact, the order complex of any graded partially ordered set is a balanced simplicial complex.)

When discussing balanced complexes, we assume that they are equipped with a fixed vertex coloring $\kappa: V(G) \rightarrow [d]$. This coloring allows us to refine the $f$- and $h$-numbers of a balanced simplicial complex.  Let $\Delta$ be a balanced simplicial complex of dimension $d-1$.  For any subset of colors $S \subseteq [d]$, define $f_S(\Delta)$ to be the number of faces in $\Delta$ for which $\kappa(F) = S$; that is, the number of faces in $\Delta$ whose vertices are colored exactly by the colors in $S$.  The numbers $f_S(\Delta)$ are called the \textit{flag $f$-numbers} of $\Delta$ and the collection $(f_S(\Delta))_{S \subseteq [d]}$ is called the \textit{flag $f$-vector} of $\Delta$.  Similarly, the \textit{flag $h$-numbers} of $\Delta$ are defined as 
\begin{equation} \label{flag-h-numbers}h_T(\Delta) = \sum_{S \subseteq T} (-1)^{|T|-|S|}f_S(\Delta) \quad \mbox{for } T\subseteq[d].\end{equation} The flag $f$-numbers refine the ordinary $f$-numbers while the flag $h$-numbers refine the ordinary $h$-numbers by the formulas $$f_{i-1}(\Delta) = \sum_{\substack{S \subseteq [d]\\ |S| = i}}f_S(\Delta) \quad \mbox{and} \quad h_j(\Delta) = \sum_{\substack{T \subseteq [d] \\ |T| = j}} h_T(\Delta).$$ Just as the $h$-numbers arise naturally in the context of the Stanley-Reisner ring of a simplicial complex, the flag $h$-numbers arise analogously for balanced simplicial complexes because the Stanley-Reisner ring can be given a refined grading according to the underlying coloring of the vertices.

\begin{example}  \label{C*d-example}
Let $\C^*_d$ denote the boundary complex of a $d$-dimensional cross-polytope.  As a simplicial complex, $\C^*_d$ has vertex set $\{u_1,u_2,\ldots,u_d,v_1,v_2,\ldots,v_d\}$ and all possible faces $F$ with the property that $|F \cap \{u_i,v_i\}| \leq 1$ for each $1 \leq i \leq d$.  In other words, $\C^*_d$ is the $d$-fold join of $\mathbb{S}^0$. Hence, $\C^*_d$ is balanced with $\kappa(u_i)=\kappa(v_i)=i$ for all $i\in [d]$, and for any $S \subseteq [d]$, an arbitrary face in $\C^*_d$ whose vertices are colored by $S$ can be described by the choice of whether it contains $u_i$ or $v_i$ for each $i \in S$.  Thus $f_S(\C^*_d) = 2^{|S|}$.  Then for any $T \subseteq [d]$, it follows from the binomial theorem that $$h_T(\C^*_d) = \sum_{S \subseteq T} (-1)^{|T|-|S|}f_S(\C^*_d) = \sum_{i=0}^{|T|}\binom{|T|}{i} (-1)^{|T|-i} 2^i = 1.$$
\end{example}

If $\Delta$ is a balanced complex and $S$ is a subset of colors, define $V_S:=\bigcup_{i\in S}V_i$ and $\Delta_S:=\Delta[V_S]$ to be the restriction of $\Delta$ to its vertices whose colors lie in the set $S$.  Also, if $\Gamma$ and $\Delta$ are balanced complexes of the same dimension, we can still define the connected sum of $\Gamma$ and $\Delta$ just as we did in the non-balanced case by adding the requirement that all vertex identifications occur between vertices of the same color; the resulting complex $\Gamma \# \Delta$ is then balanced as well. We refer to this operation as the \textit{balanced connected sum}. We define a \textit{balanced handle addition} in a similar fashion.

\subsection{Rigidity theory}
We summarize a few definitions and results from rigidity theory that will be needed for the proofs in later sections. Our presentation is mainly based on that of \cite{Kalai-rigidity}. 

Let $G=(V,E)$ be a graph. A $d$-embedding of $G$ is a map $\phi:V\rightarrow\R^d$. The distance between two $d$-embeddings $\phi$ and $\psi$ of $G$ is  $\dist(\phi,\psi):= \max_{v\in V} \dist(\phi(v),\psi(v))$, where  $\dist(-,-)$ denotes the Euclidean distance in $\R^d$. Let $\phi$ and $\psi$ be $d$-embeddings  of $G$. We say that they are \textit{isometric} if $\dist(\phi(u),\phi(v))=\dist(\psi(u),\psi(v))$ for \textit{all} $u,v\in V$. On the other hand, we say that they are \textit{$G$-isometric} if $\dist(\phi(u),\phi(v))=\dist(\psi(u),\psi(v))$ for \textit{every edge} $\{u,v\}$ of $G$. 

A $d$-embedding $\phi$ is called \textit{rigid} if there exists an $\epsilon>0$ such that every $d$-embedding $\psi$ of $G$ that is $G$-isometric to $\phi$ and satisfies $\dist(\phi,\psi)<\epsilon$ is isometric to $\phi$. Less formally, a $d$-embedding $\phi$ of $G$ is rigid if every sufficiently small perturbation of $\phi$ that preserves the lengths of the edges is induced by an isometry of $\R^d$. A graph $G$ is {\em generically $d$-rigid} if the set of rigid $d$-embeddings of $G$ is an open dense set in the set of all $d$-embeddings. 

The following result is known as the gluing lemma, see \cite[Theorem 2]{AsimowRothII} (for $d=2$) and \cite[Lemma 11.1.9]{Whiteley} (for the general case).

\begin{lemma} \label{gluing_lemma}
Let $G_1$ and $G_2$ be generically $d$-rigid graphs. If $G_1\cap G_2$ contains at least $d$ vertices, then $G_1\cup G_2$ is generically $d$-rigid.
\end{lemma}

Let $\phi$ be a $d$-embedding of $G=(V,E)$. We say that an edge $\{u,v\}$, not in $E(G)$, \textit{depends on $G$} w.r.t.~$\phi$, if for every embedding $\psi$ which is $G$-isometric to $\phi$ and also sufficiently close to $\phi$, $\dist(\psi(u),\psi(v))=\dist(\phi(u),\phi(v))$. A graph $G$ is \textit{$d$-acyclic} if for a generic $d$-embedding of $G$ no edge $\tau$ of $G$ depends on $G-\tau:=(V,E-\{\tau\})$. We also define a \textit{stress} of $G$ w.r.t.~$\phi$ as a function $w:E\to \R$ that assigns weights to the edges of $G$ in such a way that for every vertex $v\in V$,
\[
\sum_{u \, : \{v,u\}\in E} w(\{v,u\})(\phi(v)-\phi(u)) =0.
\]
We say that $G$ is \textit{generically $d$-stress free} if  $G$ has no non-zero stresses w.r.t.~a generic $d$-embedding.

The relevance of generic rigidity to the study of face numbers, first noticed and utilized by Kalai in \cite{Kalai-rigidity}, stems from the following result, see \cite[Corollary 3]{AsimowRothI} and  \cite[Section 3]{Kalai-rigidity}.

\begin{theorem} \label{rigidity_ineq}
Let $G$ be a generically $d$-rigid graph with $n$ vertices and $e$ edges. Then $e\geq dn-\binom{d+1}{2}$. Furthermore, equality $e= dn-\binom{d+1}{2}$ is attained if and only if $G$ is generically $d$-stress free, which, in turn,  happens if and only if $G$ is $d$-acyclic.
\end{theorem} 

\begin{corollary} \label{rigidity-corollary}
Let $\Delta$ be a $(d-1)$-dimensional simplicial complex.  If the graph of $\Delta$ is generically $d$-rigid, then $h_2(\Delta) \geq h_1(\Delta)$.  Moreover, equality $h_2(\Delta) = h_1(\Delta)$ holds if and only if the graph of $\Delta$ is generically $d$-stress free.
\end{corollary}

\begin{proof}
The proof is immediate as $h_2(\Delta) - h_1(\Delta) = f_1(\Delta) - df_0(\Delta) + {d+1 \choose 2}$.
\end{proof}


For the rest of the paper we will say that $\Delta$ is generically $d$-rigid or $d$-acyclic or $d$-stress free if the graph of $\Delta$ has the corresponding property.

\section{Lower bound theorems for balanced pseudomanifolds}

\subsection{History} 
Walkup \cite{Walkup} (in dimensions $3$ and $4$) and Barnette \cite{Barnette-LBT} (in all dimensions) showed that the boundary of a stacked $(d-1)$-sphere on $n$ vertices has the componentwise minimal $f$-vector among all simplicial $(d-1)$-manifolds with $n$ vertices.  An alternative proof of this result, as well as the treatment of equality cases, was given by Kalai \cite{Kalai-rigidity}. This result was later generalized by Fogelsanger \cite{Fogelsanger} and Tay \cite{Tay} to all normal pseudomanifolds.  

A stacked $(d-1)$-sphere on $n$ vertices, denoted $\ST(n,d-1)$, is defined inductively as follows.  The minimal number of vertices that a simplicial $(d-1)$-sphere can have is $d+1$, which is achieved by the boundary complex of a $d$-simplex.  Therefore, $\ST(d+1,d-1)$ is defined to be the boundary of a $d$-simplex.  Inductively, for $n>d+1$, $\ST(n,d-1)$ is defined as the connected sum of $\ST(n-1,d-1)$ with the boundary of a $d$-simplex.  Thus, $\ST(n,d-1)$ is obtained by taking the $(n-d)$-fold connected sum of the boundary of a $d$-simplex with itself.  By choosing an appropriate embedding, it is possible to realize $\ST(n,d-1)$ as the boundary of a simplicial $d$-polytope.  

The Lower Bound Theorem states that if $\Delta$ is a normal pseudomanifold of dimension $d-1\geq 2$ with $n$ vertices, then $f_i(\Delta) \geq f_i(\ST(n,d))$ for all $0 \leq i \leq d-1$.  Somewhat surprisingly, via the McMullen--Perles--Walkup reduction (see \cite{Barnette-LBT, McMullen-Walkup}), the Lower Bound Theorem can be equivalently stated as the following elegant inequality on the level of $h$-numbers.

\begin{theorem} \label{LBT} {\rm{(Lower Bound Theorem \cite{Barnette-LBT, Fogelsanger, Kalai-rigidity, Tay, Walkup}})}
Let $\Delta$ be a normal pseudomanifold of dimension $d-1\geq 2$. Then $h_2(\Delta) \geq h_1(\Delta)$. Furthermore, if $d\geq 4$, then equality holds if and only if $\Delta$ is a stacked sphere. 
\end{theorem}

As an extension of this result, Goff, Klee, and Novik \cite{Goff-Klee-Novik} defined the family of stacked cross-polytopal spheres to give a lower bound theorem for balanced spheres.  Just as the boundary of a $d$-simplex has the minimal $h$-numbers (and hence the minimal $f$-numbers) among all simplicial $(d-1)$-spheres, it follows easily from the results in \cite{Stanley-balanced} that the boundary of a $d$-dimensional cross-polytope has the minimal (flag) $h$-numbers among all balanced simplicial $(d-1)$-spheres.  For any pair of integers $n$ and $d$ with $n$ divisible by $d$, a \textit{stacked cross-polytopal $(d-1)$-sphere on $n$ vertices}, denoted $\ST^{\times}(n,d-1)$, is defined as the balanced connected sum of $\frac{n}{d}-1$ copies of $\C^*_d$ with itself.  

Goff, Klee, and Novik \cite{Goff-Klee-Novik} (for simplicial spheres and more generally for doubly Cohen-Macaulay complexes) and Browder and Klee \cite{Browder-Klee} (for simplicial manifolds and more generally for Buchsbaum$^*$ simplicial complexes) showed that when $n$ is divisible by $d$, $\ST^{\times}(n,d-1)$ has the componentwise minimal $f$-vector among all balanced spheres/manifolds of dimension $d-1$ with $n$ vertices.  Once again, this can be stated as a simple inequality on the level of $h$-numbers.  

\begin{theorem} \label{balancedLBT-manifolds} {\rm{(Balanced Lower Bound Theorem \cite{Goff-Klee-Novik, Browder-Klee})}}
Let $\Delta$ be a balanced connected simplicial manifold of dimension $d-1 \geq 2$. Then  $2h_2(\Delta) \geq (d-1)h_1(\Delta)$.  
\end{theorem}

The $h$-number inequality in the balanced Lower Bound Theorem motivates the following definition of the balanced $g$-numbers of a balanced simplicial complex. 

\begin{definition}
Let $\Delta$ be a balanced simplicial complex of dimension $d-1$.  For any $1 \leq j \leq d$, define the \textit{balanced $g$-number} $$\overline{g}_j(\Delta):=j\cdot h_j(\Delta) - (d-j+1)\cdot h_{j-1}(\Delta).$$  
\end{definition}

Just as the classic Lower Bound Theorem states that $g_2(\Delta) \geq 0$ for a normal pseudomanifold, the balanced Lower Bound Theorem states that $\overline{g}_2(\Delta) \geq 0$ for a balanced connected simplicial manifold. Note also that if $j=2$ and $d=3$, then $\overline{g}_2(\Delta)=2g_2(\Delta)$, and so in this case, the inequality $\overline{g}_2\geq 0$ of Theorem \ref{balancedLBT-manifolds} is equivalent to the inequality $g_2\geq 0$ of Theorem \ref{LBT}.  

\subsection{The balanced LBT}
Our first goal in this section is to use rigidity theory to show that the balanced Lower Bound Theorem continues to hold for balanced normal pseudomanifolds.  

\begin{theorem}  \label{balancedLBT}
Let $\Delta$ be a balanced, normal $(d-1)$-pseudomanifold with $d \geq 3$.  Then $\overline{g}_2(\Delta) \geq 0$.  
\end{theorem}  

The structure of the proof of Theorem \ref{balancedLBT} is similar to that of \cite[Theorem 5.3]{Goff-Klee-Novik} and \cite[Theorem 4.1]{Browder-Klee}. The main new ingredient is the following result that might be of interest in its own right.

\begin{lemma} \label{3-rigidity}
Let $\Delta$ be a balanced, normal $(d-1)$-pseudomanifold with $d \geq 3$, and let $S$ be a $3$-element subset of $[d]$. Then $\Delta_S$ is generically $3$-rigid. 
\end{lemma}

\begin{proof}
The proof is by induction on $d$. For $d=3$, $\Delta_S=\Delta$, and the lemma follows from the main result of \cite{Fogelsanger} asserting that  any normal 2-pseudomanifold is generically 3-rigid.  Thus assume that $d>3$. Since $\Delta$ is pure,
\[
\Delta_S= \bigcup_{\substack{F\in\Delta_{[d]-S}\\ |F| = d-3}} \lk_\Delta(F).
\]
In addition, since $\Delta$ is a normal pseudomanifold, each link in the above union is a normal 2-pseudomanifold, and hence is generically 3-rigid by the Fogelsanger's result. Thus by the gluing lemma (Lemma \ref{gluing_lemma}), to complete the proof it only remains to show that we can order the facets of $\Delta_{[d]-S}$ as $F_0, F_1, \ldots, F_s$ in such a way that  
\begin{equation}  \label{order}
\lk_\Delta(F_i)\bigcap \left(\bigcup_{j<i} \lk_\Delta(F_j)\right) \mbox{ contains a facet of }\Delta_S \quad \forall 1\leq i \leq s.
\end{equation}

To do so, pick any facet $H$ of $\Delta$. Since $\Delta$ is a normal pseudomanifold, every facet $G$ of $\Delta$ has a finite distance to $H$, denoted $\dist(H,G)$, in the facet-ridge graph of $\Delta$. (For instance, the distance of $H$ to itself is $0$; the distance of $H$ to every facet that shares a $(d-2)$-face with $H$ is $1$, etc.). Order the facets of $\Delta$ as $H=H_0, H_1,\ldots, H_r$ so that $\dist(H,H_k)\leq \dist(H,H_{\ell})$ for all $1\leq k\leq \ell \leq r$. Since $\Delta$ is pure, each facet of $\Delta_{[d]-S}$ is the intersection of some facet of $\Delta$ with $V_{[d]-S}$. Hence the above ordering induces an ordering on the facets of $\Delta_{[d]-S}$: simply consider the list $H_0\cap V_{[d]-S}, H_1\cap V_{[d]-S},\ldots, H_r\cap V_{[d]-S}$, and for each element $F$ in this list, delete all the occurrences of $F$ except the very first one.

We claim that this ordering satisfies condition (\ref{order}). Indeed, let $F_i$, $i\geq 1$, be
a facet of $\Delta_{[d]-S}$,  let $p$ be the least index such that $F_i=H_p\cap  V_{[d]-S}$, and let $t=\dist(H,H_p)$. Since the facet-ridge graph of $\Delta$ is connected, there exists $0\leq p'<p$ such that $\dist(H, H_{p'})=t-1$ and $\dist(H_p, H_{p'})=1$ (or, equivalently, $|H_p\cap H_{p'}|=d-1$). Since $p$ is minimal, $H_{p'}\cap V_{[d]-S} \neq F_i$, and so $H_{p'}\cap V_{[d]-S}=F_j$ for some $j<i$. As $H_p$ and $H_{p'}$ differ in only one element, we conclude that $H_p\cap V_S=H_{p'}\cap V_S$. This intersection is therefore a facet of $\Delta_S$ contained in $\lk_\Delta(F_i)\cap (\cup_{j<i} \lk_\Delta(F_j))$. 
\end{proof}

We are now ready to complete the proof of Theorem \ref{balancedLBT}.

\smallskip\noindent\textit{Proof of Theorem \ref{balancedLBT}: \ }
It was observed in the proof of \cite[Theorem 5.3]{Goff-Klee-Novik} that
 $(d-2)h_2(\Delta)=\sum_{|S|=3} h_2(\Delta_S)$ while  $\binom{d-1}{2}h_1(\Delta)=\sum_{|S|=3} h_1(\Delta_S)$. Hence
\begin{equation}  \label{g2_as_a_sum}
\overline{g}_2(\Delta)=\frac{2}{d-2} \sum_{|S|=3} (h_2(\Delta_S)-h_1(\Delta_S))=
\frac{2}{d-2}\sum_{|S|=3} g_2(\Delta_S).
\end{equation}
By Lemma \ref{3-rigidity},  $\Delta_S$ (for $|S|=3$) is generically $3$-rigid. Thus, by Corollary \ref{rigidity-corollary}, the last sum in (\ref{g2_as_a_sum}) is the sum of non-negative numbers, and so it is nonnegative. The result follows. \hfill $\square$

\subsection{Inequalities involving  the first Betti number}

For connected orientable homology manifolds (i.e., manifolds with non-vanishing top homology), the Lower Bound Theorem can be strengthened as follows.

\begin{theorem} \label{thm:Novik-Swartz-lower-bound} {\rm{\cite[Theorem 5.2]{Novik-Swartz-buchsbaum}}}
Let $\Delta$ be a connected simplicial complex  of dimension $d-1 \geq 3$. If $\Delta$ is a $\field$-homology manifold that is orientable over $\field$, then $g_2(\Delta)\geq \binom{d+1}{2} \beta_1(\Delta;\field)$.  
\end{theorem}

\begin{remark}
During the time that this paper was under review, Murai \cite{Murai-15} proved Theorem \ref{thm:Novik-Swartz-lower-bound} for all normal pseudomanifolds.  For this reason, it is natural to pose the following Conjecture \ref{b1_conj} and Conjecture \ref{b1_conj_rev} for balanced normal pseudomanifolds as well.
\end{remark}


Is there a balanced analog of Theorem \ref{thm:Novik-Swartz-lower-bound}? We conjecture that the following holds.

\begin{conjecture} \label{b1_conj}
Let $\Delta$ be a balanced connected simplicial complex  of dimension $d-1 \geq 3$. If $\Delta$ is a $\field$-homology manifold that is orientable over $\field$, then $\overline{g}_2(\Delta) \geq 4\binom{d}{2} \beta_1(\Delta;\field)$.
\end{conjecture}

The constant $4\binom{d}{2}$ is explained by the following result that provides a partial evidence to Conjecture \ref{b1_conj}. (Note that this result holds in the generality of normal pseudomanifolds.)

\begin{theorem}  \label{t-covering}
Let $\Delta$ be a balanced normal pseudomanifold of dimension $d-1\geq 2$. If $\|\Delta\|$ has a connected $t$-sheeted covering space, then $\overline{g}_2(\Delta) \geq 4\frac{t-1}{t}\binom{d}{2}$. In particular, if $\beta_1(\Delta;\Q) \neq 0$, then $\overline{g}_2(\Delta) \geq 4\binom{d}{2}$.
\end{theorem}

\begin{proof}
The proof follows the same ideas as the proof of \cite[Theorem 4,3]{Swartz}. Let $X:=\|\Delta\|$ and let $X^t$ be a connected $t$-sheeted covering space of $X$. Then  the triangulation $\Delta$ of $X$ lifts to a triangulation $\Delta^t$ of $X^t$. The balancedness of $\Delta$ implies that $\Delta^t$ is balanced as well (just color each vertex $v$ of $\Delta^t$ in the color of the image of $v$ in $\Delta$). Similarly, our assumption that $\Delta$ is a normal $(d-1)$-pseudomanifold implies that $\Delta^t$ is also a normal $(d-1)$-pseudomanifold. (For instance, to see the connectivity of links of $\Delta^t$, observe that if $F^t$ is a face of $\Delta^t$ whose image is $F\in\Delta$ and $p^t$ is a point in the interior of $F^t$ whose image is $p$, then $\widetilde{H_0}(\lk_{\Delta^t}(F^t))\cong H_{|F|}(X^t, X^t-p^t)\cong H_{|F|}(X, X-p)\cong \widetilde{H_0}(\lk_{\Delta}(F))$.) 

To finish the proof, apply Theorem \ref{balancedLBT} to $\Delta^t$ to obtain that 
\begin{equation} \label{LBT-for-Delta^t}
\overline{g}_2(\Delta^t)=2h_2(\Delta^t) - (d-1)h_1(\Delta^t)\geq 0.
\end{equation}
On the other hand, according to \cite[Proposition 4.2]{Swartz}, 
\begin{equation} \label{h_i(Delta^t)}
h_i(\Delta^t)=th_i(\Delta)+(-1)^{i-1}(t-1)\binom{d}{i} \quad \mbox{for } i=1,2.
\end{equation}
Substituting eq.~(\ref{h_i(Delta^t)}) in eq.~(\ref{LBT-for-Delta^t}) gives
\begin{eqnarray*}
\overline{g}_2(\Delta^t) &=& 2h_2(\Delta^t) - (d-1) h_1(\Delta^t) \\
&=& 2\left( t h_2(\Delta)-(t-1){d \choose 2}\right) - (d-1) \left(t h_1(\Delta) + (t-1){d \choose 1}\right) \\
&=& t\left(2h_2(\Delta)-(d-1)h_1(\Delta)\right) - 4(t-1){d \choose 2}\\
&=& t\overline{g}_2(\Delta)- 4(t-1){d \choose 2}.
\end{eqnarray*}
The claim that $\overline{g}_2(\Delta) \geq 4 \frac{t-1}{t}{d \choose 2}$ follows from the fact that $\overline{g}_2(\Delta^t) \geq 0$.This inequality together with the well-known fact that if  $\beta_1(\Delta;\Q) \neq 0$, then $\|\Delta\|$ has a connected $t$-sheeted covering space for arbitrarily large $t$ implies the ``in particular" part. \end{proof} 


\section{Extremal cases and the balanced Walkup class}

The main goal of this section is to establish a balanced analog of the second part of the Lower Bound Theorem concerning the cases of equality. 

\begin{theorem} \label{balancedLBT-equality}
Let $\Delta$ be a balanced normal $(d-1)$-pseudomanifold with $d\geq 4$. Then $\overline{g}_2(\Delta)=0$ holds if and only if $\Delta$ is a stacked cross-polytopal sphere.
\end{theorem}

If $\Delta$ is a stacked cross-polytopal sphere, the result follows from the fact that $h_i(A\#B) = h_i(A)+h_i(B)$ for all $0<i<d$ and the simple calculation  $h_i(\C^*_d)=\binom{d}{i}$ for $0\leq i\leq d$. The proof that $\overline{g}_2(\Delta) = 0$ implies that $\Delta$ is a stacked cross-polytopal sphere is similar in spirit to one of the proofs of \cite[Theorem 7.1]{Kalai-rigidity} (see Sections 8 and 9 of \cite{Kalai-rigidity}) and will take a large part of this section. It will require the following lemmas and definitions.

A \textit{missing face} of a simplicial complex $\Delta$ is any subset $F$ of the vertex set of $\Delta$ with the property that $F$ is not a face of $\Delta$ but every proper subset of $F$ is. The dimension of a missing face $F$ is $|F|-1$. A \textit{missing $i$-face} is a missing face of dimension $i$. A simplicial complex $\Delta$ is a \textit{flag complex} if all of its missing faces have dimension 1.

\begin{lemma} \label{restrictions}
Let $\Delta$ be a balanced normal $(d-1)$-pseudomanifold with $d\geq 3$. Then $\overline{g}_2(\Delta)=0$ if and only if  $g_2(\Delta_S)=0$  for every $3$-element subset $S$ of $[d]$, which in turn happens if and only if $\Delta_S$ is generically $3$-stress free for every $3$-element subset $S$ of $[d]$.
\end{lemma}
\begin{proof} This is an immediate consequence of eq.~(\ref{g2_as_a_sum}), Lemma \ref{3-rigidity}, and Corollary \ref{rigidity-corollary}.
\end{proof}

\begin{lemma} \label{links}
Let $\Delta$ be a balanced normal $(d-1)$-pseudomanifold with $d\geq 4$. If $\overline{g}_2(\Delta)=0$, then $\overline{g}_2(\lk_\Delta(v))=0$ for every vertex $v$ of $\Delta$.
\end{lemma}
\begin{proof} Pick a vertex $v$ of $\Delta$ and let $c$ be the color of $v$. Then for every $3$-element subset $T$ of $[d]-\{c\}$, $(\lk_\Delta(v))_T$ is a subcomplex of $\Delta_T$. Since $\overline{g}_2(\Delta)=0$, by Lemma \ref{restrictions}, $\Delta_T$ is generically 3-stress free. As a subgraph of a generically 3-stress free graph is also generically 3-stress free, we conclude that $(\lk_\Delta(v))_T$ is generically 3-stress free for every subset $T\subseteq [d]-\{c\}$ of size 3. Lemma \ref{restrictions} applied to $\lk_\Delta(v)$ then completes the proof. (Note that since $d\geq 4$, $d-1=\dim(\lk_\Delta(v))+1\geq 3$.) \end{proof}

\begin{lemma} \label{missing_faces}
Let $\Delta$ be a balanced normal $(d-1)$-pseudomanifold with $d\geq 3$. If $\overline{g}_2(\Delta)=0$, then every missing face of $\Delta$ has dimension $1$ or $d-1$.
\end{lemma}
\begin{proof} For $d=3$ there is nothing to prove, so assume that $d\geq 4$ and that $F=\{u_0, u_1,\ldots, u_k\}$ is a missing $k$-face of $\Delta$ for some $1<k<d-1$. Since $k+1<d$, there is a vertex $z$ of $\Delta$ whose color is different from the colors of vertices of $F$. Let $T$ be the color set of $\{u_0,u_1,z\}$, let $U=\{u_0,u_1\}\in\Delta_T$, and let $W=F-U$. As $F$ is a missing face of $\Delta$, it follows that $U$ is a missing 1-face of $\lk_\Delta(W)$, and hence  also of $(\lk_\Delta(W))_T$. However, by Lemma \ref{3-rigidity}, the graph of  $(\lk_\Delta(W))_T$ is generically 3-rigid. Hence the edge $U$ depends on the graph of $(\lk_\Delta(W))_T$ (w.r.t.~a generic embedding), and so $(\lk_\Delta(W))_T \cup\{U\}$ is not generically 3-stress free. Since $(\lk_\Delta(W))_T \cup\{U\}$ is a subcomplex of $\Delta_T$, we obtain that $\Delta_T$ is not generically $3$-stress free. This  contradicts Lemma~\ref{3-rigidity}.
\end{proof}

\begin{lemma}  \label{missing(d-1)face}
Let $\Delta$ be a balanced connected simplicial $(d-1)$-manifold with $d\geq 3$. Assume further that either $d\in\{3,4\}$ or all vertex links of $\Delta$ are stacked cross-polytopal spheres. If $\overline{g}_2(\Delta)=0$ and $\Delta$ has a missing $(d-1)$-face, then there exist balanced connected simplicial $(d-1)$-manifolds $\Delta_1$ and $\Delta_2$ such that $\Delta$ is the balanced connected sum of $\Delta_1$ and $\Delta_2$; moreover $\overline{g}_2(\Delta_1)= \overline{g}_2(\Delta_2)=0$.
\end{lemma}
\begin{proof} Let $F$ be a missing $(d-1)$-face of $\Delta$.  Let $\Gamma$ be the simplicial complex obtained from $\Delta$ by cutting along the boundary of $F$ and filling in the two missing $(d-1)$-faces that result from $F$. Then $\Gamma$ is a simplicial manifold that is either connected or is a disjoint union of two complexes $\Delta_1$ and $\Delta_2$. In the former case, $\Delta$ is obtained from $\Gamma$ by adding a handle, and in the latter case, $\Delta=\Delta_1\#\Delta_2$. (These results are well-known: the operation of cutting and patching is introduced by Walkup in \cite{Walkup}, where he also proves the above statements for $d=4$, see \cite[Lemma 3.2]{Walkup}, and a certain variant of these statements for $d>4$, \cite[Lemma 4.2]{Walkup}. A much more general statement for all $d\geq 3$ is proved in \cite[Lemma 3.3]{BagchiDatta-LBT}. )

Furthermore, since $F$ is a missing $(d-1)$-face of $\Delta$, the vertices of $F$ form a $d$-clique in the graph of $\Delta$, and hence no two of them have the same color. Thus, $\Delta\cup \{F\}$ is a balanced complex. Consequently, $\Gamma$ is balanced, and $\Delta$ is obtained from $\Gamma$ either by a balanced handle addition or as a balanced connected sum. In the former case, since $f_1(\Gamma)=f_1(\Delta)+\binom{d}{2}$ and $f_0(\Gamma)=f_0(\Delta)+d$, we obtain that $\overline{g}_2(\Gamma)=\overline{g}_2(\Delta)-4\binom{d}{2}=-4\binom{d}{2}<0$, which is impossible by Theorem \ref{balancedLBT}. Therefore, $\Delta$ is the balanced connected sum of $\Delta_1$ and $\Delta_2$. Then $0=\overline{g}_2(\Delta)=\overline{g}_2(\Delta_1)+\overline{g}_2(\Delta_2)$, and hence $\overline{g}_2(\Delta_1)= \overline{g}_2(\Delta_2)=0$ holds by Theorem \ref{balancedLBT}.
\end{proof}

\begin{lemma} \label{cross_polytope}
Let $\Delta$ be a balanced normal $(d-1)$-pseudomanifold with $d\geq 3$. If all vertex links of $\Delta$ are boundary complexes of a cross-polytope, then so is $\Delta$.
\end{lemma}
\begin{proof} Let $n$ be the number of vertices of $\Delta$. Since all vertex links of $\Delta$ are isomorphic to $\C^*_{d-1}$, every vertex of $\Delta$ has exactly $2(d-1)$ neighbors. Hence, $f_1(\Delta)=(d-1)n$. Therefore, $h_1(\Delta)=n-d$,
$h_2(\Delta)=f_1(\Delta)-(d-1)h_1(\Delta)-\binom{d}{2}=\binom{d}{2}$, and $\overline{g}_2(\Delta)=2h_2(\Delta)-(d-1)h_1(\Delta)=(d-1)(2d-n)$. However, by Theorem \ref{balancedLBT}, $\overline{g}_2(\Delta)\geq 0$, and so $n\leq 2d$. On the other hand, it is well-known and easy to prove (e.g., by induction on $d$; cf.~Proposition \ref{at_least_3d} below) that every balanced $(d-1)$-pseudomanifold has at least $2d$ vertices and the only one that has exactly $2d$ vertices is $\C^*_d$. 
\end{proof}
%

We are now in a position to complete the proof of Theorem \ref{balancedLBT-equality}.  The proof will be by induction on $d$.  In the base case that $d=4$, Lemma \ref{missing_faces} tells us that either $\Delta$ is flag or $\Delta$ has a missing $3$-face.  We require two additional lemmas in the former case.

\begin{lemma} \label{3-manifold-lemma1}
Let $\Delta$ be a simplicial $3$-manifold that is both flag and balanced and satisfies $\overline{g}_2(\Delta)=0$.  Assume $v_1$ and $v_2$ are two vertices of the same color such that $\lk_{\Delta}(v_1) \cap \lk_{\Delta}(v_2)$ contains a $2$-face, $F$.  Then either $\lk_{\Delta}(v_1)$ and $\lk_{\Delta}(v_2)$ intersect precisely along  the simplex on $F$ or $\lk_{\Delta}(v_1) = \lk_{\Delta}(v_2)$, in which case $\Delta$ is the suspension with $v_1$ and $v_2$ as suspension vertices.
\end{lemma}

\begin{proof}
Assume without loss of generality that $v_1$ and $v_2$ have color 4.  Let $n_1$ (respectively $n_2$) denote the number of vertices in $\lk_{\Delta}(v_1)$ (resp. the link of $v_2$), and let $k$ denote the number of vertices in $\lk_{\Delta}(v_1) \cap \lk_{\Delta}(v_2)$.  Each $\lk_{\Delta}(v_i)$ is a (homology) $2$-sphere, and hence has $3n_i-6$ edges.  

By Lemma \ref{restrictions}, since $\overline{g}_2(\Delta) = 0$, the graph of $\Delta_{[3]}$ is generically $3$-stress free.  Further, since $\lk_{\Delta}(v_1)$ and $\lk_{\Delta}(v_2)$ share a $2$-face, the union $\Gamma:= \lk_{\Delta}(v_1) \cup \lk_{\Delta}(v_2)$ is also generically $3$-rigid by the gluing lemma (Lemma \ref{gluing_lemma}).  Since $\Gamma$ is a subcomplex of $\Delta_{[3]}$, it must also be generically $3$-stress free.  Therefore $\Gamma$ has $3(n_1+n_2-k)-6$ edges. Since the link of each $v_i$ has $3n_i-6$ edges, it follows that $\lk_{\Delta}(v_1) \cap \lk_{\Delta}(v_2)$ has $3k-6$ edges.  If $k=3$, then $\lk_{\Delta}(v_1) \cap \lk_{\Delta}(v_2)$ is the simplex on $F$.  Otherwise, $k \geq 4$ and the graph of $\lk_{\Delta}(v_1) \cap \lk_{\Delta}(v_2)$ is  planar (because $\lk_{\Delta}(v_1)$ and $\lk_{\Delta}(v_2)$ have planar graphs) with $k$ vertices and $3k-6$ edges.  Hence the graph of $\lk_{\Delta}(v_1) \cap \lk_{\Delta}(v_2)$ is the graph of a simplicial $2$-sphere.  

Since $\Delta$ is flag, $\lk_{\Delta}(v_i)$ is the restriction of $\Delta$ to the neighbors of $v_i$ for $i=1,2$.  Therefore, $\lk_{\Delta}(v_1) \cap \lk_{\Delta}(v_2)$ is the restriction of $\Delta$ to the vertices in $\lk_{\Delta}(v_1) \cap \lk_{\Delta}(v_2)$.  Hence $\lk_{\Delta}(v_1)$ and $\lk_{\Delta}(v_2)$ contain $\lk_{\Delta}(v_1) \cap \lk_{\Delta}(v_2)$ as an induced triangulated $2$-sphere.  Since $\lk_{\Delta}(v_1)$ and $\lk_{\Delta}(v_2)$ are $2$-spheres on their own, it must be the case that $\lk_{\Delta}(v_1) = \lk_{\Delta}(v_2)$. The statement follows.  
\end{proof}

\begin{lemma} \label{3-manifold-lemma2}
Let $\Delta$ be a simplicial $3$-manifold that is both flag and balanced.  Assume further that $\overline{g}_2(\Delta) = 0$.  Then $\Delta$ is isomorphic to $\mathcal{C}^*_4$.
\end{lemma}

\begin{proof}
It suffices to prove that $\Delta$ has two vertices of each color.  Assume to the contrary that $\Delta$ has at least three vertices of some color (w.l.o.g., say color $4$).   We construct a graph, $H$, on the set of vertices of color $4$ in $\Delta$  with an edge connecting two vertices whose links intersection contains a $2$-face.  Hence this intersection \textit{is} a $2$-face by Lemma \ref{3-manifold-lemma1}.

First we claim that $H$ does not contain any cycles.  Suppose to the contrary that $C = v_0, v_1, \ldots, v_{\ell}, v_0$ is a cycle of minimal length in $H$.  As in the proof of Lemma \ref{3-manifold-lemma1}, let $n_i$ denote the number of vertices in $\lk_{\Delta}(v_i)$.  Fix $0 < j \leq \ell$, and let $\Gamma_j:= \lk_{\Delta}(v_0) \cup \cdots \cup \lk_{\Delta}(v_{j-1})$.  

Let $m$ denote the number of vertices in $\Gamma_j$ and let $k$ denote the number of vertices in $\lk_{\Delta}(v_j) \cap \Gamma_j$.  As in the proof of Lemma \ref{3-manifold-lemma1}, the graph of $\Gamma_j$ is generically $3$-rigid and generically $3$-stress free.  Similarly, the graph of $\lk_{\Delta}(v_j)$ is generically $3$-rigid and generically $3$-stress free.  Therefore $\Gamma_j$ has $3m-6$ edges, $\lk_{\Delta}(v_j)$ has $3n_j-6$ edges, and hence $\lk_{\Delta}(v_j) \cap \Gamma_j$ has $3k-6$ edges. Since $\lk_{\Delta}(v_j)$ is a $2$-sphere, as in the proof of Lemma \ref{3-manifold-lemma1}, it follows that either $\lk_{\Delta}(v_j)$ intersects  $\Gamma_j$ along a single $2$-simplex or $\lk_{\Delta}(v_j)$ is contained in $\Gamma_j$.

We claim that it is impossible to have $\lk_{\Delta}(v_j) \subseteq \Gamma_j$. Indeed, if $j<\ell$, then by our assumption that $v_{j+1}$ is adjacent to $v_j$ in $H$, there exists a $2$-face, $F$, such that $\lk_{\Delta}(v_j)$ intersects $\lk_{\Delta}(v_{j+1})$ along the simplex on $F$.  If $\lk_{\Delta}(v_j) \subseteq \Gamma_j$, then there exists $i<j$ such that $F \in \lk_{\Delta}(v_i)$. This however is impossible as it would imply that a ridge $F$ of $\Delta$ is contained in at least three facets of $\Delta$, namely, $F\cup\{v_i\}$, $F\cup\{v_j\}$, and $F\cup\{v_{j+1}\}$. On the other hand, if $j=\ell$, and $\lk_{\Delta}(v_\ell) \subseteq \Gamma_\ell$, then since $f_2(\lk_\Delta(v_\ell))>2$ and since $\lk_{\Delta}(v_\ell)$ shares a single $2$-face with $\lk_{\Delta}(v_0)$ and a single $2$-face with $\lk_{\Delta}(v_{\ell-1})$, it must be that $\ell>2$ and that $\lk_{\Delta}(v_\ell)$ shares a face with $\lk_{\Delta}(v_i)$ for some $0<i<\ell-1$. Hence $v_\ell$ is adjacent to $v_{i}$ in $H$. This implies that the cycle $C$ could have been shortened, which contradicts the minimality of $C$.

 As a consequence, $\lk_{\Delta}(v_j)$ intersects $\Gamma_j$ along a single $2$-simplex for all $0<j \leq \ell$.  However, $\lk_{\Delta}(v_\ell)$ intersects $\lk_{\Delta}(v_0)$ and $\lk_{\Delta}(v_{\ell-1})$ along different $2$-faces,  and so $\lk_{\Delta}(v_\ell)$ intersects $\Gamma_{\ell}$ in more than a single $2$-face, which creates a contradiction.  

Therefore, $H$ does not contain any cycles.  The number of vertices in $H$ is equal to $|V_4|$; the number of vertices in $\Delta$ of color $4$.  Therefore, the number of edges in $H$ is at most $|V_4|-1$.  On the other hand, the number of edges in $H$ is equal to $\frac{1}{2} \sum_{v \in V_4}f_2(\lk_{\Delta}(v))$ as the neighbors of a vertex $v$ in $H$ are in one-to-one correspondence with the $2$-faces in $\lk_{\Delta}(v)$.  The link of each vertex in $\Delta$ is a $2$-sphere, and hence has at least $4$ two-dimensional faces.  Therefore, we have $$2|V_4| \leq \frac{1}{2} \sum_{v \in V_4}f_2(\lk_{\Delta}(v)) \leq |V_4|-1,$$ which is absurd. Thus, $\Delta$ can only have two vertices of each color and hence must be $\mathcal{C}^*_4$. \end{proof}

We need one more reduction to deal with the base case of $d=4$. Recall that $\widetilde{\chi}(\Delta):=\sum_{i=0}^{d-1}(-1)^i \beta_i(\Delta;\field)$ denotes the \textit{reduced Euler characteristic} of $\Delta$, and that the family of simplicial $3$-manifolds coincides with that of homology $3$-manifolds.
\begin{lemma} \label{orientable}
Let $\Delta$ be a balanced normal pseudomanifold of dimension $3$. If $\overline{g}_2(\Delta)=0$, then $\Delta$ is a simplicial manifold.
\end{lemma}
\begin{proof} If $\Delta$ is not a simplicial manifold, then there is a vertex $v$ of $\Delta$ such that $\lk_\Delta(v)$ is a $2$-dimensional normal pseudomanifold that is not a sphere. Then $\lk_\Delta(v)$ is a 2-dimensional simplicial manifold with $\widetilde{\chi}(\lk_{\Delta}(v))<1$. Hence by the Dehn-Sommerville relations, $$\overline{g}_2(\lk_\Delta(v))=2\left(h_2(\lk_\Delta(v))-h_1(\lk_\Delta(v))\right)=-6(\widetilde{\chi}(\lk_\Delta(v)-1)>0.$$ This contradicts Lemma \ref{links}.
\end{proof} 

\noindent \textit{Proof of Theorem \ref{balancedLBT-equality}:} Let $\Delta$ be a balanced normal $(d-1)$-pseudomanifold with $d\geq 4$ and $\overline{g}_2(\Delta)=0$. We prove that $\Delta$ is a  stacked cross-polytopal sphere by induction on $d$. 

We start with the case of $d=4$. Then, by Lemma \ref{orientable}, $\Delta$ is a simplicial 3-manifold. Our proof in this case is by induction on the number of vertices. If $\Delta$ has a missing $3$-face, then by Lemma \ref{missing(d-1)face}, $\Delta$ is the balanced connected sum of balanced simplicial 3-manifolds $\Delta_1$ and $\Delta_2$ with $\overline{g}_2(\Delta_1)= \overline{g}_2(\Delta_2)=0$. By the induction hypothesis, $\Delta_1$ and $\Delta_2$ are stacked cross-polytopal spheres, and hence so is $\Delta$. By Lemma \ref{missing_faces}, we thus can assume that $\Delta$ is a flag complex. Lemma \ref{3-manifold-lemma2} then yields that $\Delta = \mathcal{C}^*_4$.


Assume now that $d>4$. Then by Lemma \ref{links}, for every vertex $v$ of $\Delta$, $\overline{g}_2(\lk_\Delta(v))=0$. Hence, by the induction hypothesis on $d$, all vertex links of $\Delta$ are stacked cross-polytopal spheres. We proceed by induction on the number of vertices of $\Delta$. If all vertex links of $\Delta$ are the boundaries of a cross-polytope, then by Lemma \ref{cross_polytope}, $\Delta$ itself is $\C^*_d$, and we are done. Otherwise, there is a vertex $v$ whose link is  the connected sum of at least two copies of $\C^*_{d-1}$. Then the link of $v$ has a missing $(d-2)$-face, $F'$. Let $F=F'\cup\{v\}$. Since, by Lemma \ref{missing_faces}, $\Delta$ has no missing $(d-2)$-faces, we infer that $F'$ is a face of $\Delta$, and hence that $F$ is a missing $(d-1)$-face of $\Delta$. Lemma \ref{missing(d-1)face} then guarantees that $\Delta$ is the balanced connected sum of simplicial $(d-1)$-manifolds $\Delta_1$ and $\Delta_2$ with $\overline{g}_2(\Delta_1)= \overline{g}_2(\Delta_2)=0$. By the induction hypothesis on the number of vertices, $\Delta_1$ and $\Delta_2$ are stacked cross-polytopal spheres, and hence so is $\Delta$. \hfill $\square$ \newline


Theorem \ref{balancedLBT-equality} raises a question of ``Which balanced manifolds have only stacked cross-polytopal spheres as their vertex links?" The answer turns out to be completely analogous to the non-balanced case and requires the following 
definition.  We refer to \cite{Walkup} and \cite{Kalai-rigidity} for the definition of the (non-balanced) Walkup class.
\begin{definition} The \textit{balanced Walkup class}, $\BH^d$, consists of balanced $(d-1)$-dimensional complexes that are obtained from the boundary complexes of $d$-dimensional cross-polytopes by successively applying the operations of balanced connected sums and balanced handle additions.
\end{definition}

Note that if $\Delta\in \BH^d$ then $|V_1(\Delta)|=|V_2(\Delta)|=\cdots=|V_d(\Delta)|$, and so $f_0(\Delta)$ is a multiple of $d$. Also, since vertex links of $\Delta_1\#\Delta_2$ (resp.~$\Gamma^\phi$) are connected sums of vertex links of $\Delta_1$ and $\Delta_2$ (resp.~$\Gamma$), it follows that for an arbitrary element $\Delta$ of $\BH^d$, all vertex links of $\Delta$ are stacked cross-polytopal spheres. If $d\geq 5$, then the converse also holds. More precisely, we have the following balanced analog of \cite[Theorem 8.3, Corollary 8.4]{Kalai-rigidity}:

\begin{theorem} \label{balanced8.3}
Let $\Delta$ be a balanced normal $(d-1)$-pseudomanifold with $d\geq 4$. If all vertex links of $\Delta$ are cross-polytopal spheres and if $\Delta$ has no missing $(d-2)$-faces, then $\Delta\in\BH^d$.
\end{theorem}

\begin{corollary} \label{balanced8.4}
Let $\Delta$ be a balanced normal $(d-1)$-pseudomanifold with $d\geq 5$. If all vertex links of $\Delta$ are cross-polytopal spheres then $\Delta\in\BH^d$.
\end{corollary}

Since the proofs are also very similar to those in \cite{Kalai-rigidity}, we only sketch the main ideas. To derive Corollary \ref{balanced8.4} from Theorem \ref{balanced8.3}, observe that stacked cross-polytopal spheres of dimension $d-2$ only have missing faces of dimension 1 and $d-2$. Consequently, if $\Delta$ satisfies conditions of Corollary \ref{balanced8.4}, then it has no missing $(d-2)$-faces (as existence of such a missing face $F$ would force a missing $(d-3)$-face in $\lk_\Delta(v)$ for $v\in F$), and hence $\Delta\in \BH^d$ by Theorem \ref{balanced8.3}.  \hfill $\square$

The proof of Theorem  \ref{balanced8.3} relies on the following easy fact. We leave its verification to our readers and remark that its non-balanced counterpart can be found in \cite[Lemma 4.8]{BagchiDatta-LBT}  
\begin{lemma} \label{easy} Let $\Delta_1$ and $\Delta_2$ be balanced normal $(d-1)$-pseudomanifolds with $d\geq 3$, and let $\Delta=\Delta_1\#\Delta_2$ be obtained from $\Delta_1$ and $\Delta_2$ by the balanced connected sum. If $\Delta$ is a stacked cross-polytopal sphere, then so are $\Delta_1$ and $\Delta_2$.
\end{lemma}

\smallskip\noindent\textit{Proof of Theorem \ref{balanced8.3}: \ } If all vertex links of $\Delta$ are the boundaries of the cross-polytope, then so is $\Delta$ (see Lemma \ref{cross_polytope}), and we are done. Otherwise, there exists a vertex $v$ of $\Delta$ such that $\lk_\Delta(v)$ is the connected sum of at least two copies of $\C^*_{d-1}$, and hence $\lk_\Delta(v)$ has a missing $(d-2)$-face, $F'$. Let $F=F'\cup\{v\}$. Since $\Delta$ has no missing $(d-2)$-faces, $F'$ is a face of $\Delta$, and we conclude that $F$ is a missing $(d-1)$-face of $\Delta$. As in the proof of Lemma \ref{missing(d-1)face}, cut along $F$ and patch with two $(d-1)$-simplices. It follows from our assumptions on $\Delta$ and from Lemma \ref{easy} that all vertex links of the resulting simplicial manifold are cross-polytopal spheres. The result then follows by the double induction on $f_0$ and $\overline{g}_2$. \hfill $\square$

\smallskip

Since (as an easy computation shows) adding a balanced handle to a balanced $(d-1)$-dimensional complex causes $\overline{g}_2$ to increase by $4\binom{d}{2}$, the elements of $\BH^d$ (for $d-1\geq 3$) satisfy the inequality of Conjecture \ref{b1_conj} as equality. This observation along with \cite[Theorem 5.2]{Novik-Swartz-buchsbaum}, \cite[Theorem 1.14]{Bagchi}, and \cite[Corollary 3.15]{Datta-Murai} suggest the following strengthening of Conjecture \ref{b1_conj}.

\begin{conjecture} \label{b1_conj_rev}
Let $\Delta$ be a balanced connected $\field$-homology manifold that is orientable over $\field$ of dimension $d-1 \geq 3$.  Then $\overline{g}_2(\Delta) \geq 4\binom{d}{2} \beta_1(\Delta;\field)$ and equality holds if and only if $\Delta\in\BH^d$.
\end{conjecture}
\section{A balanced generalized lower bound conjecture}

\subsection{History}

McMullen and Walkup \cite{McMullen-Walkup} proposed the following Generalized Lower Bound Conjecture (GLBC, for short) as an extension of the classic Lower Bound Theorem.  

\begin{conjecture}
Let $P$ be a simplicial $d$-polytope.  Then 
\begin{enumerate}
\item $g_j(P) \geq 0$ for all $1 \leq j \leq \halffloor$, and
\item for any $1 \leq r \leq \halffloor$, the following are equivalent: 
\begin{enumerate}
\item $g_r(P)=0$;
\item $P$ is $(r-1)$-stacked; i.e., there exists a triangulation of $P$, all of whose faces of dimension at most $d-r$ are faces of $P$. 
\end{enumerate}
\end{enumerate}
\end{conjecture}

This conjecture is now a theorem: Part 1 of the GLBC follows from the $g$-theorem;  McMullen and Walkup \cite{McMullen-Walkup} showed that $g_r(P) = 0$ whenever $P$ is an $(r-1)$-stacked simplicial polytope;  finally, Murai and Nevo \cite{Murai-Nevo} completed the proof of the GLBC by showing that if $g_r(P) = 0$, then $P$ must be $(r-1)$-stacked. 

In this section, we propose an extension of these results to the family of balanced simplicial polytopes.  First, we require a definition that extends the notion of an $(r-1)$-stacked polytope to the family of balanced polytopes.  

\subsection{$(r-1)$-stackedness for balanced spheres and manifolds}
We start by defining a certain subclass of regular CW complexes. 

\begin{definition} \label{cross-polytopal-complex} A $d$-dimensional regular CW complex $\X$ is called a \textit{cross-polytopal} complex if the $(d-1)$-dimensional skeleton of $\X$ is a simplicial complex, the boundary of every $d$-dimensional cell of $\X$ coincides with $\C^*_d$, and no two $d$-dimensional cells have the same vertex set (equivalently, no two $d$-cells share their entire boundary).
\end{definition}

If $\X$ is a $d$-dimensional simplicial complex or a cross-polytopal complex, then we say that $\X$ is a \textit{$\field$-homology manifold with boundary} if (i) for each $p\in\|\X\|$, the pair $(\|\X\|,\|\X\|-p; \field)$ has the relative homology of a $d$-ball or a $d$-sphere, and (ii) \textit{the boundary complex of $\X$}, \[\partial\X:=\{F\in \X \,: \, H_*(\|\X\|,\|\X\|-p; \field)\cong 0 \mbox{ for $p$ in the relative interior of $F$}\}\cup\{\emptyset\},\] is a closed homology $(d-1)$-manifold. Similarly, we say that $\X$ is a \textit{homology $d$-ball} if (i) $\X$ is a homology manifold with boundary, (ii) $H_*(\|\X\|; \field)\cong 0$, and (iii) the boundary complex of $\X$ is a homology $(d-1)$-sphere. Note that the boundary complex of $\X$ is contained in the $(d-1)$-skeleton of $\X$, and hence is a simplicial complex. If $\X$ is a homology manifold with boundary, then every face of $\X$ that does not belong to $\partial\X$ is called an \textit{interior} face.

\begin{definition} Let $\Delta$ be a balanced $(d-1)$-dimensional homology sphere (resp.~connected homology manifold without boundary), and let $1\leq r \leq d$. We say that $\Delta$ has the \textit{balanced $(r-1)$-stacked property} if there exists a cross-polytopal complex $\X$ such that (i) $\X$ is a homology $d$-ball (resp.~homology $d$-manifold with boundary), (ii) the boundary complex of $\X$ is $\Delta$, and (iii) all faces of $\X$ of dimension at most $d-r$ are faces of $\Delta$. Such a CW complex $\X$ is called an \textit{$(r-1)$-stacked cross-polytopal decomposition of $\Delta$}.
\end{definition}

\begin{remark}  \label{alternate}
An alternate definition of the balanced $(r-1)$-stacked property states that there exists a $d$-dimensional simplicial complex $\Gamma$ such that (i) $\Gamma$ is a homology $d$-ball (or $d$-manifold with boundary); (ii) $\Gamma$ is balanced; (iii) the boundary complex of $\Gamma$ is $\Delta$; (iv) all faces of $\Gamma$ of dimension at most $d-r$ that do not use a vertex of color $d+1$ are faces of $\Delta$; and (v) if $u$ is a vertex of color $d+1$, then $u$ belongs to the interior of $\Gamma$ and $\lk_{\Gamma}(u)$ is isomorphic to $\mathcal{C}^*_d$.  

Indeed, if $\mathcal{X}$ is an $(r-1)$-stacked cross-polytopal decomposition of $\Delta$, we can perform a stellar subdivision of each of its cross-polytopal faces, which will introduce new vertices that receive color $d+1$.  This perspective will be useful in some of our subsequent proofs.
\end{remark}

\begin{remark} The only balanced homology $(d-1)$-manifold that has the balanced 0-stacked property is $\C^*_d$. A balanced homology sphere has the balanced 1-stacked property if and only if it is a stacked cross-polytopal sphere, and a balanced homology $(d-1)$-manifold has the balanced 1-stacked property if and only if it belongs to $\BH^d$. (This last statement follows easily from methods/results in \cite{Datta-Murai} and is a balanced analog of \cite[Corollary 3.12]{Datta-Murai}.) The suspension of a stacked cross-polytopal sphere is an example of a balanced sphere that has the balanced 2-stacked property. More generally, if $\Delta$ has the balanced $(r-1)$-stacked property, then the suspension of $\Delta$ has the balanced $r$-stacked property.
\end{remark}

Somewhat informally, we say that a polytope $P$ is balanced if $P$ is a simplicial polytope whose boundary complex is a balanced complex. We also say that $P$ has the balanced $(r-1)$-stacked property if the boundary complex of $P$ does. This leads to the following balanced Generalized Lower Bound Conjecture (balanced GLBC, for short) that extends the balanced LBT.

\begin{conjecture}  \label{balancedGLBC}
Let $P$ be a balanced $d$-polytope.  Then 
\begin{enumerate}
\item $\overline{g}_j(P) \geq 0$ for all $1 \leq j \leq \halffloor$, and
\item for an $1 \leq r \leq \halffloor$, the following are equivalent: 
\begin{enumerate}
\item $\overline{g}_r(P) = 0$; 
\item $P$ has the balanced $(r-1)$-stacked property. 
\end{enumerate}
\end{enumerate}
\end{conjecture}

\noindent As with the $g$-conjecture, it is tempting to propose the above conjecture in the generality of arbitrary balanced simplicial (or even homology) spheres. 

Just as the first part of the classic GLBT can be written as 
\begin{equation} \label{simplex_ineq}
1 = h_0(P) \leq h_1(P) \leq \cdots \leq h_{\halffloor}(P),
\end{equation}
 the first part of the balanced GLBC is equivalent to the requirement that 
\begin{equation} \label{cross-poly-ineq}
1 = \frac{h_0(P)}{\binom{d}{0}} \leq \frac{h_1(P)}{\binom{d}{1}} \leq \frac{h_2(P)}{\binom{d}{2}} \leq \cdots \leq \frac{h_{\halffloor}(P)}{\binom{d}{\halffloor}}.
\end{equation}
 This follows by observing that the inequality $\overline{g}_j(P) \geq 0$ is equivalent to the inequality $\binom{d}{j-1} h_j(P) \geq \binom{d}{j}h_{j-1}(P)$, which served as the motivation for our definition of the balanced $g$-numbers.
It is also worth pointing out that the similarity between equations (\ref{simplex_ineq}) and (\ref{cross-poly-ineq}) goes deeper: if we denote by $\sigma_d$ the boundary complex of the $d$-simplex, then (\ref{simplex_ineq}) and (\ref{cross-poly-ineq}) can be rewritten as
\[
 \frac{h_0(P)}{h_0(\sigma_d)} \leq \frac{h_1(P)}{h_1(\sigma_d)} \leq \cdots \leq \frac{ h_{\halffloor}(P)}{h_{\halffloor}(\sigma_d)} \quad \mbox{and} \quad
\frac{h_0(P)}{h_0(\C^*_d)} \leq \frac{h_1(P)}{h_1(\C^*_d)} \leq \cdots  \leq \frac{ h_{\halffloor}(P)}{ h_{\halffloor}(\C^*_d)}, \mbox{ respectively.}
\] 

\begin{remark}
In the time that this paper was under review, Juhnke-Kubitzke and Murai \cite{Kubitzke-Murai} established the nonnegativity of the $\overline{g}$-numbers of a balanced $d$-polytope, hence settling the first part of the balanced GLBC. 
\end{remark}

\begin{remark} Bj\"orner and Swartz \cite[Problem 4.2]{Swartz-higherCM} conjectured that any $2i$-dimensional doubly Cohen--Macaulay complex satisfies the inequality $h_i\leq h_{i+1}$. As rank selected subcomplexes of a balanced homology $(d-1)$-sphere $\Delta$ are doubly Cohen--Macaulay, the conjecture by Bj\"orner and Swartz, if true, would imply that $h_i(\Delta_T)\leq h_{i+1}(\Delta_T)$ for all $(2i+1)$-element subsets $T$ of $[d]$. Since a routine double-counting argument (analogous to that of \cite[Theorem 5.3]{Goff-Klee-Novik}) shows that $$\sum_{\substack{T \subseteq [d] \\ |T| = 2i+1}}h_{i}(\Delta_T)=\binom{d-i}{i+1}h_i(\Delta) \quad \mbox{and} \quad \sum_{\substack{T \subseteq [d] \\ |T| = 2i+1}}h_{i+1}(\Delta_T)=\binom{d-i-1}{i}h_i(\Delta),$$
and since $\binom{d-i-1}{i}/\binom{d-i}{i+1}=(i+1)/(d-i)$, their conjecture would imply Part 1 of Conjecture \ref{balancedGLBC} even for the class of all balanced homology spheres.
\end{remark}

\begin{remark} \label{h_3-vs-h_1} If $\Delta$ is a balanced normal $(d-1)$-pseudomanifold with $d\geq 4$, then by Theorem \ref{balancedLBT}, each vertex $v$ of $\Delta$ satisfies $\overline{g}_2(\lk_\Delta)(v)\geq 0$, and hence 
\[2\sum_{v\in V(\Delta)} h_2(\lk_\Delta(v))-(d-2)\sum_{v\in V(\Delta)} h_1(\lk_\Delta(v))\geq 0.\]
By \cite[Proposition 2.3]{Swartz-higherCM}, $\sum_{v\in V(\Delta)} h_i(\lk_\Delta(v))=(i+1)h_{i+1}(\Delta)+(d-i)h_i(\Delta)$ for all $0\leq i\leq d-1.$
Plugging the latter equation (with $i=1$ and $i=2$) into the former, we conclude that $6h_3(\Delta)-(d-1)(d-2)h_1(\Delta)\geq 0$. Thus, for every  balanced normal $(d-1)$-pseudomanifold with $d\geq 4$, at least the inequality $\frac{h_1(\Delta)}{\binom{d}{1}} \leq \frac{h_3(\Delta)}{\binom{d}{3}}$ holds. Moreover,  $\frac{h_1(\Delta)}{\binom{d}{1}} = \frac{h_3(\Delta)}{\binom{d}{3}}$ if and only if all vertex links satisfy $\overline{g}_2(\lk_\Delta(v))=0$.
\end{remark}

The first step towards proving the classic GLBT was McMullen-Walkup's proof that $(b) \Longrightarrow (a)$ in Part 2 of the statement of the theorem.  We will show that the analogous result holds for the balanced GLBC. 

\begin{theorem}\label{thm:balancedGLBC-2b-implies-2a}
Let $\Delta$ be a balanced homology $(d-1)$-sphere with the balanced $(r-1)$-stacked property.  Then $\overline{g}_j(\Delta) = 0$ for all $r \leq j \leq \halffloor$. 
\end{theorem}

The proof of Theorem \ref{thm:balancedGLBC-2b-implies-2a} requires several lemmas and some notation.  

\begin{notation} \label{stellar-subdiv}
Let $\Delta$ be a balanced homology $(d-1)$-sphere with the balanced $(r-1)$-stacked property, and let $\X$ be an $(r-1)$-stacked cross-polytopal decomposition of $\Delta$. Let  $C_1, C_2,\ldots, C_m$ be the collection of the (closed) $d$-cells of $\X$. As in Remark \ref{alternate}, we consider a stellar subdivision of each $C_i$: for $1\leq i\leq m$, let $b_i$ be a new vertex and let $b_i *\partial C_i$ be the cone over the boundary of $C_i$ with apex $b_i$. Replacing each $C_i$ in $\X$ with $b_i *\partial C_i$ creates a \textit{simplicial} complex $\Gamma$ with the property that $\|\Gamma\|=\|\X\|$ and $\partial\Gamma=\partial\X=\Delta$. Moreover, no face of $\Gamma$ contains two of the vertices among $\{b_1,b_2,\ldots,b_m\}$, and so $\Gamma$ can be realized as a balanced homology $d$-ball by assigning a new color $d+1$ to the new vertices $b_1,b_2,\ldots,b_m$.
\end{notation}
 
We begin by proving a general lemma about balanced homology $d$-balls with the property that every face that contains a vertex of color $d+1$ is an interior face. If $\Gamma$ is a homology ball, we write $f_i(\Gamma^{\circ})$ to denote the number of interior $i$-faces in $\Gamma$ so that $f_i(\Gamma) = f_i(\Gamma^{\circ}) + f_i(\partial\Gamma)$. Similarly, when $\Gamma$ is balanced we write $f_S(\Gamma^{\circ})$ to denote the number of interior faces in $\Gamma$ whose vertices are colored by the colors in $S$ and define $h_T(\Gamma^{\circ}) = \sum_{S \subseteq T}(-1)^{(|T|-|S|)}f_S(\Gamma^{\circ})$.  

\begin{lemma} \label{lem:balanced-symmetry}
Let $\Gamma$ be a balanced homology $d$-ball with the property that each face $F \in \Gamma$ that contains a vertex of color $d+1$ is an interior face.  Let $T \subseteq [d+1]$.
\begin{enumerate}
\item If $d+1\notin T$, then $h_T(\Gamma)=h_T(\Gamma^{\circ})+h_T(\partial \Gamma)$.
\item If $d+1\in T$, then $h_{[d+1]-T}(\Gamma)=h_T(\Gamma^{\circ})$.
\end{enumerate}
\end{lemma}
\begin{proof} For Part 1, note that since $\Gamma$ is the disjoint union of $\Gamma^{\circ}$ and $\partial\Gamma$, $f_S(\Gamma)=f_S(\Gamma^{\circ})+f_S(\partial\Gamma)$ for any $S \subseteq [d]$. The assertion then follows from the defining relation for the flag $h$-numbers, see eq.~(\ref{flag-h-numbers}). 

For Part 2, consider a new vertex $b_0$, of color $d+1$, and let $\Lambda:=\Gamma \cup (b_0*\partial\Gamma)$. Since $\partial\Gamma$ does not contain vertices of color $d+1$, it follows that $\Lambda$ is a balanced homology $d$-sphere. Also since $\Lambda$ is the disjoint union of $\Gamma^{\circ}$ and $b_0*\partial\Gamma$, $f_S(\Lambda)=f_S(\Gamma^{\circ})+f_S(b_0*\partial\Gamma)$ for all $S\subseteq [d+1]$. Thus 
\begin{equation}  \label{h_T(Lambda)}
h_T(\Lambda)= \sum_{S\subseteq T} (-1)^{|T|-|S|}f_S(\Gamma^\circ)+\sum_{S\subseteq T} (-1)^{|T|-|S|}f_S(b_0*\partial\Gamma).
\end{equation}

We now consider two cases according to whether or not $d+1$ belongs to $T$. If $d+1\notin T$, then for all $S\subseteq T$, $f_S(b_0*\partial\Gamma)=f_S(\partial\Gamma)$, and the above equation simplifies to
\begin{equation}  \label{(d+1)notinT}
h_T(\Lambda)=h_T(\Gamma^\circ)+h_T(\partial\Gamma)=h_T(\Gamma) \quad \mbox{if } d+1\notin T,
\end{equation}
where the last step is by Part 1 of the lemma.
On the other hand, if $d+1\in T$, then 
\begin{eqnarray*}
\sum_{S\subseteq T} &(-1)^{|T|-|S|}& f_S(b_0*\partial\Gamma) \\ &=&
\sum_{S\subseteq T-\{d+1\}} \left[(-1)^{|T|-|S|}f_S(b_0*\partial\Gamma) + 
(-1)^{|T|-|S|-1}f_{S\cup\{d+1\}}(b_0*\partial\Gamma)\right] \\
&=& \sum_{S\subseteq T-\{d+1\}} (-1)^{|T|-|S|} \left[f_S(\partial\Gamma)-f_S(\partial\Gamma)\right]=0,
\end{eqnarray*}  
and so equation (\ref{h_T(Lambda)}) reduces to 
\begin{equation} \label{(d+1)inT}
h_T(\Lambda)=h_T(\Gamma^\circ) \quad \mbox{if } d+1\in T.
\end{equation}

The statement of the lemma now follows from the generalized Dehn-Sommerville relations for balanced homology spheres \cite{BayerBillera} asserting that $h_T(\Lambda)=h_{[d+1]-T}(\Lambda)$. Indeed, if $d+1\in T \subseteq [d+1]$, then
\[
h_T(\Gamma^\circ)=h_T(\Lambda)=h_{[d+1]-T}(\Lambda)=h_{[d+1]-T}(\Gamma),
\]
where the first step is by eq.~(\ref{(d+1)inT}), the second is by the generalized Dehn-Sommerville relations, and the last one is by  eq.~(\ref{(d+1)notinT}).
\end{proof}

Similar computations show that Part 2 of Lemma \ref{lem:balanced-symmetry} also holds for  $d+1\notin T\subseteq [d+1]$, and that if $d+1\in T \subseteq [d+1]$, then $h_T(\Gamma)=h_T(\Gamma^{\circ})-h_T(\partial \Gamma)$. We do not include a proof since we will not use these cases.

\begin{lemma} \label{lem:stacked-interior-hnums}
Let $\Delta$ be a balanced homology $(d-1)$-sphere with the balanced $(r-1)$-stacked property, and let $\X$ and $\Gamma$ be as in Notation \ref{stellar-subdiv}. Let $T \subseteq [d+1]$ be a subset of colors.  
\begin{enumerate} 
\item If $d+1 \notin T$ and $|T| \leq d+1-r$, then $h_T(\Gamma^{\circ}) = 0$. 
\item If $d+1 \in T$ and $|T| \leq d+2-r$, then $h_T(\Gamma^{\circ}) = m$, where $m=f_d(\X)$ is the number of cross-polytopes in an $(r-1)$-stacked cross-polytopal decomposition of $\Delta$. 
\end{enumerate}
\end{lemma}
\begin{proof}
Recall that $$h_T(\Gamma^{\circ}) = \sum_{S \subseteq T}(-1)^{|T|-|S|}f_S(\Gamma^{\circ}).$$  Let $F$ be a face of $\Gamma$ with $|F| \leq d-r+1$, so $\dim F\leq d-r$.  If $F$ does not contain a vertex of color $d+1$, then $F$ is a face of $\X$, and hence $F\in\skel_{d-r}(\X)$. However, since $\Delta$ has the balanced $(r-1)$-stacked property, $\skel_{d-r}(\X)=\skel_{d-r}(\Delta)$. Thus $F$ is a face of $\Delta$, and so it is a boundary face of $\Gamma$. In other words, if $|S| \leq d-r+1$ and $d+1 \notin S$, then $f_S(\Gamma^{\circ}) = 0$.  This proves Part 1.  

Now assume that $T$ satisfies the assumptions of Part 2 of the lemma.  If $S \subseteq T$ and $d+1 \notin S$, then $|S| \leq |T|-1 \leq d+1-r$.  Therefore the above argument shows that $f_S(\Gamma^{\circ}) = 0$.  Thus 
\begin{equation}\label{eqn4}
h_T(\Gamma^{\circ}) = \sum_{\substack{S \subseteq T \\ d+1 \in S}} (-1)^{|T|-|S|}f_S(\Gamma^{\circ}). 
\end{equation}

Next, notice that if $S \subseteq [d+1]$ and $d+1 \in S$, then $f_S(\Gamma^{\circ}) = m \cdot 2^{|S|-1}$.   Indeed, if $G$ is a face whose vertices are colored by $S$, then $G$ contains some vertex $b_i$.  Consider $G' = G - \{b_i\}$.  For each cross-polytope $C_i$, there are $2^{|S|-1}$ such faces $G'$, and hence $f_S(\Gamma^{\circ}) = m \cdot 2^{|S|-1}$.  Thus Part 2 follows from the binomial theorem and eq.~\eqref{eqn4} exactly as in Example \ref{C*d-example}.  \end{proof}

Theorem \ref{thm:balancedGLBC-2b-implies-2a} is a consequence of the following result, which shows that certain $h$-numbers of a homology $(d-1)$-sphere with the balanced $(r-1)$-stacked property depend only on the number of cross-polytopes in its cross-polytopal decomposition. 

\begin{theorem} \label{thm:balancedGLBC-computation}
Let $\Delta$ be a balanced homology $(d-1)$-sphere with the balanced $(r-1)$-stacked property, let $\X$ and $\Gamma$ be as in Notation \ref{stellar-subdiv}, and let $m=f_d(\X)$. If $T \subseteq [d]$ and $r-1 \leq |T| \leq d-r+1$, then $h_T(\Delta) = h_T(\partial\Gamma) = m$.  
\end{theorem}

\begin{proof} For $T$ satisfying the assumptions of the theorem and $S: = [d+1] - T$, we have 

$\bullet$ $h_T(\Gamma^{\circ}) = 0$. (This follows from Lemma \ref{lem:stacked-interior-hnums}, since $d+1 \notin T$ and $|T| \leq d+1-r$.) 

$\bullet$  $h_T(\Gamma^{\circ}) = h_T(\Gamma) - h_T(\partial\Gamma)$.  (This is a consequence of Part 1 of  Lemma \ref{lem:balanced-symmetry}.)

$\bullet$  $h_T(\Gamma) = h_{S}(\Gamma^{\circ})$. (See Part 2 of Lemma \ref{lem:balanced-symmetry}; recall that $T\subseteq [d]$, and hence $d+1\in S$.)

$\bullet$  $h_S(\Gamma^{\circ}) = m$.  (See Part 2 of Lemma \ref{lem:stacked-interior-hnums}; note that  $d+1 \in S$ and $|S| = d+1-|T| \leq d+1-(r-1) = d+2-r$.)

Putting these four results together gives
\[
0= h_T(\Gamma^{\circ}) = h_T(\Gamma) - h_T(\partial\Gamma) =
h_S(\Gamma^{\circ})-h_T(\partial\Gamma) = m - h_T(\partial\Gamma).
\] The result follows.
\end{proof}

\noindent \textit{Proof of Theorem \ref{thm:balancedGLBC-2b-implies-2a}:}
For $\Delta$ as in the statement of the theorem and any $r-1 \leq j \leq d-r+1$, Theorem \ref{thm:balancedGLBC-computation} tells us that,
\begin{equation} \label{h_j-vs-m}
h_j(\Delta) = \sum_{\substack{T \subseteq [d] \\ |T| = j}}h_T(\partial\Gamma) = \sum_{\substack{T \subseteq [d] \\ |T| = j}} m = \binom{d}{j} \cdot m.
\end{equation}
Thus, for $r\leq j\leq \halffloor$,  $\overline{g}_j(\Delta) = j\cdot h_j(\Delta) - (d-j+1)\cdot h_{j-1}(\Delta) = j \cdot \binom{d}{j} \cdot m - (d-j+1) \cdot \binom{d}{j-1} \cdot m = 0.$
\hfill $\square$

\begin{remark} Equation (\ref{h_j-vs-m}) is completely analogous to the non-balanced situation: if $\Delta$ is an $(r-1)$-stacked homology $(d-1)$-sphere and $B$ is a homology $d$-ball such that $\partial B=\Delta$ and $\skel_{d-r}(B)=\skel_{d-r}(\Delta)$, then it is not hard to show that $h_j(\Delta)=f_d(B)$ for all $r-1 \leq j \leq d-r+1$.
\end{remark}

\subsection{Uniqueness of $(r-1)$-stacked cross-polytopal decomposition}
The second step towards proving the classical GLBT was McMullen's observation \cite{McMullen04} that if an $(r-1)$-stacked triangulation of a simplicial polytope exists, then it is unique. An explicit description of such a triangulation was then worked out in \cite[Theorem 2.3]{Murai-Nevo} (for $(r-1)$-stacked spheres) and \cite[Theorem 2.20]{Bagchi-Datta-stellated} and  \cite[Theorem 4.2]{Murai-Nevo2} (for $(r-1)$-stacked manifolds). In the rest of this section we establish balanced analogs of these results. 

\begin{definition} \label{Delta(j)} If $\Delta$ is a balanced $(d-1)$-dimensional complex and $1\leq j<d$, then define $\overline{\Delta}(j)$ to be the following $d$-dimensional cross-polytopal complex: 
\begin{itemize}
\item The $(d-1)$-skeleton of $\overline{\Delta}(j)$, denoted $\Delta(j)$, consists of all simplices $F\subset V$ such that $\skel_j(F)\subseteq \Delta$.
\item For each $2d$-subset  $W=\{u_1,v_1,\ldots, u_d, v_d\}$ of $V$ with the property that  $\{u_i,v_i\}\subseteq V_i(\Delta)$ for all $i=1,\ldots, d$ and $\Delta(j)[W]\cong \C^*_d$, there is a $d$-cell (cross-polytope) attached to $\Delta(j)[W]$.
\end{itemize}
\end{definition}

Note that $\Delta\subseteq \Delta(j)$, and so $\Delta(j)$ is at least $(d-1)$-dimensional. On the other hand, since $\Delta$ is balanced, the graph of $\Delta$ contains no $(d+1)$-cliques. The assumption that $j\geq 1$ then implies that $\Delta(j)$ is exactly $(d-1)$-dimensional. Note also that $\overline{\Delta}(j)\supseteq \overline{\Delta}(j+1)$ for all $1\leq j <d-1$, and that $\overline{\Delta}(j)$ is the maximal cross-polytopal complex with the property that $\skel_j(\overline{\Delta}(j))=\skel_j(\Delta)$.

The significance of $\overline{\Delta}(j)$ is explained by the following theorem that is a balanced analog of the aforementioned results by Murai--Nevo and Bagchi--Datta; its proof is similar in spirit to that of  \cite[Theorem 2.3]{Murai-Nevo} and \cite[Theorem 2.20]{Bagchi-Datta-stellated}, but requires some additional twists and bookkeeping. 

\begin{theorem} \label{uniqueness_thm}
Let $\Delta$ be a balanced $(d-1)$-dimensional connected homology manifold with the balanced $(r-1)$-stacked property, and let $\X$ be an $(r-1)$-stacked cross-polytopal decomposition of~$\Delta$.
\begin{enumerate}
\item If $\Delta$ is a homology sphere and $2\leq r\leq (d+1)/2$, then $\X=\overline{\Delta}(r-1)=\overline{\Delta}(d-r)$.
\item If $\Delta$ is a homology manifold and  $2\leq r\leq d/2$, then $\X=\overline{\Delta}(r)=\overline{\Delta}(d-r)$.
\end{enumerate}
\end{theorem}

The proof of Part 1 of Theorem \ref{uniqueness_thm} relies on the following lemma.

\begin{lemma} \label{no_missing_faces}
Let $\Delta$ be a balanced $(d-1)$-dimensional $\field$-homology sphere with the balanced $(r-1)$-stacked property (with $r\geq 2$), and let $\X$ be  an $(r-1)$-stacked cross-polytopal decomposition of $\Delta$. Then 
\begin{enumerate}
\item $\X$ has no missing simplicial faces of dimension $\geq r$.
\item $\X$ has no missing cross-polytopal faces, that is, if  $W=\{u_1,v_1,\ldots, u_d, v_d\}\subseteq V(\Delta)$ is such that $(\skel_{d-1}\X)[W]\cong \C^*_d$, then $\X[W]$ is one of the closed $d$-cells of $\X$.
\end{enumerate}
\end{lemma}
\begin{proof} Throughout this proof we use the $d$-cells $C_i$ ($i=1,\ldots, m$), the vertices $b_i$ of color $d+1$, and the $d$-dimensional simplicial complex $\Gamma$  as defined in Notation \ref{stellar-subdiv}. Also, as in the proof of Lemma \ref{lem:balanced-symmetry}, we introduce one more vertex, $b_0$, of color $d+1$ and define $\Lambda=\Gamma\cup(b_0*\Delta)$. Then $\Lambda$ is a balanced $\field$-homology $d$-sphere, and $V(\Lambda)=V(\Delta)\amalg V_{d+1}$, where $V_{d+1}=\{b_0, b_1,\ldots, b_m\}$. 

Let $F\subseteq V(\Delta)=V(\X)$ be a set of size $\ell+1$, where $\ell\ge r\ge 2$, such that all proper subsets of $F$ are faces of $\X$; in particular $|F|\leq d$ (by balancedness of $\Delta$). For Part 1, we must show that $F\in\X$, or, equivalently, that $\beta_{\ell-1}(\X[F];\field)=0$. Since $\skel_{d-1}(\X)$ is an induced subcomplex of both $\Gamma$ and $\Lambda$, $\X[F]=\Gamma[F]=\Lambda[F]$. Thus, by Alexander duality, 
\[
\beta_{\ell-1}(\X[F];\field)=\beta_{\ell-1}(\Lambda[F];\field)=\beta_{d-\ell}(\Lambda[V(\Lambda)-F];\field),
\]
and so to complete the proof of Part 1, it suffices to show that $\beta_{d-\ell}(\Lambda[V(\Lambda)-F];\field)=0$.

Set $\Lambda_i:=\Lambda[\{b_0,b_1,\ldots, b_i\}\cup V(\Delta)-F]$ for $0\leq i \leq m$. We will prove by induction on $i$ that $\beta_{d-\ell}(\Lambda_i;\field)=0$. As $\Lambda_m=\Lambda[V(\Lambda)-F]$, this will imply Part 1 of the lemma.

Since $d-\ell\leq d-r$ and since $\X$ is an $(r-1)$-stacked cross-polytopal decomposition of $\Delta$,  we have
\[
\skel_{d-\ell}(\Lambda_0)=\skel_{d-\ell}(b_0*\Delta[V(\Delta)-F]).
\]
Also, by definition of $\Lambda$, $\Lambda_0\supseteq b_0*\Delta[V(\Delta)-F]$. Hence, by definition of simplicial homology and since $b_0*\Delta[V(\Delta)-F]$ is a cone, we obtain that 
\[
\beta_{d-\ell}(\Lambda_0;\field)\leq \beta_{d-\ell}(b_0*\Delta[V(\Delta)-F];\field) =0.
\]
This establishes the  $i=0$ case of the induction.

For the induction step notice that
\[
\Lambda_{i+1}=\Lambda_i \bigcup \left(b_{i+1}*(\partial C_{i+1})[V(C_{i+1})-F]\right)
\quad \mbox{and} \quad \Lambda_{i+1}\cap\Lambda_i=(\partial C_{i+1})[V(C_{i+1})-F].
\]
In other words, to obtain $\Lambda_{i+1}$ from $\Lambda_i$, we attach to $\Lambda_i$ a \textit{cone} that intersects $\Lambda_i$ along a $(d-1)$-sphere or a $(d-1)$-ball: the former happens if $F$ is disjoint from $C_{i+1}$, as in this case, $\Lambda_{i+1}\cap\Lambda_i=\partial C_{i+1}\cong \C^*_d$, while the latter happens if $F$ intersects $C_{i+1}$, as in this case $\Lambda_{i+1}\cap\Lambda_i$ is the antistar of a non-empty face in $\C^*_d$. In either case, since $\ell\geq r\geq 2$ and since all but the top homology of the $(d-1)$-sphere (resp.~ball) vanish, a simple application of the Mayer-Vietoris sequence shows that $\beta_{d-\ell}(\Lambda_{i+1};\field)=\beta_{d-\ell}(\Lambda_i;\field)=0$. Part 1 of the lemma follows.

For Part 2, assume that $\X[W]$ is not a $d$-cell of $\X$. Then $\Lambda[W]=(\skel_{d-1}\X)[W]\cong \C^*_d$. Therefore, by Alexander duality,
\[
1=\beta_{d-1}(\Lambda[W];\field)=\beta_{0}(\Lambda[V(\Lambda)-W];\field),
\]
and so the graph of $\Lambda[V(\Lambda)-W]$ must be disconnected. On the other hand, all vertices of $V(\Delta)-W$ are connected to $b_0$ (by definition of $\Lambda$). In addition, since $\X[W]$ (by our assumption) is not one of the cells $C_i$ of $\X$, each $C_i$ ($i=1,\ldots, m$) has at least one vertex in $V(\Delta)-W$, and hence $b_i$ is connected to that vertex. This shows that the graph of $\Lambda[V(\Lambda)-W]$ is connected (in fact, has diameter 2), which is a contradiction. Part 2 follows. \end{proof}

We are now in a position to prove Theorem \ref{uniqueness_thm}. For Part 2 of the theorem, we need to extend the notion of vertex links from simplicial complexes to cross-polytopal complexes. If $C$ is a $d$-dimensional cell (cross-polytope) whose boundary is $\{u_1,v_1\}*\cdots*\{u_d,v_d\}$, then the link of $u_1$ in $C$, $\lk_C(u_1)$, is defined as the $(d-1)$-cell whose boundary is $\{u_2,v_2\}*\cdots*\{u_{d},v_{d}\}$. (The links of other vertices are defined in a similar way.) If $\X$ is a $d$-dimensional cross-polytopal complex, $v$ is a vertex of $\X$, and $C_{i_1},\ldots, C_{i_k}$ are the $d$-cells that contain $v$, then define the \textit{link of $v$ in $\X$}, $\lk_\X(v)$, as $\lk_{\skel_{d-1}(\X)}(v)$ (this is the ``simplicial" part of the link) together with all the $(d-1)$-cells $\lk_{C_{i_j}}(v)$, $j=1,\ldots, k$. Note that if $\|\X\|$ is a homology $d$-manifold with boundary, and $v$ is \textit{not} an interior vertex of $\X$, then for every two $d$-cells $C, C'$ containing $v$, the vertex sets of $\lk_C(v)$ and $\lk_{C'}(v)$ are not the same, and so $\lk_\X(v)$ is a $(d-1)$-dimensional cross-polytopal complex (in the sense of Definition \ref{cross-polytopal-complex}). \newline

\noindent\textit{Proof of Theorem \ref{uniqueness_thm}: \ } For Part 1, since $\X$ is an $(r-1)$-stacked cross-polytopal decomposition of $\Delta$, $\skel_{d-r}(\X)=\skel_{d-r}(\Delta)=\skel_{d-r}(\overline{\Delta}(d-r))$. As $r-1\leq d-r$, Lemma \ref{no_missing_faces} (along with Definition \ref{Delta(j)}) then implies that $\X=\overline{\Delta}(d-r)$ and that $\overline{\Delta}(d-r)$ has no missing faces of dimension $\geq r$. Thus we also have that $\overline{\Delta}(d-r)=\overline{\Delta}(r-1)$, and Part 1 follows.

For Part 2, let $r\leq d/2$, let $\Delta$ be a balanced $(d-1)$-dimensional homology manifold, and let $\X$ be an $(r-1)$-stacked cross-polytopal decomposition of $\Delta$. By Definition \ref{Delta(j)}, we then have $\X\subseteq \overline{\Delta}(d-r)\subseteq \overline{\Delta}(r)$, and to complete the proof it remains to show that $\X\supseteq  \overline{\Delta}(r)$. To do so, we rely on Part 1.

Let $v$ be a vertex of $\X$. Then $\lk_\X(v)$ is a $(d-1)$-dimensional cross-polytopal complex, and since $\skel_{d-r}(\X)=\skel_{d-r}(\Delta)$, we infer that 
\[\skel_{(d-1)-r}(\lk_\X(v))=\skel_{(d-1)-r}(\lk_\Delta(v)).
\]
Thus $\lk_\X(v)$ is an $(r-1)$-stacked cross-polytopal decomposition of $\lk_\Delta(v)$. As $\lk_\Delta(v)$ is a balanced homology $(d-1)$-sphere and $r\leq \frac{d}{2}=\frac{(d-1)+1}{2}$, Part 1 of the theorem yields
\begin{equation} \label{overline_link}
\lk_\X(v)=\overline{\lk_\Delta(v)}(r-1) \qquad \mbox{for all } v\in V(\X)=V(\Delta).
\end{equation}

We claim that 
\begin{equation} \label{more_overline_links}
\lk_{\overline{\Delta}(r)}(v)\subseteq \overline{\lk_\Delta(v)}(r-1) \qquad \mbox{for all } v\in V(\X)=V(\Delta).
\end{equation}
Indeed, let $F\in \lk_{\overline{\Delta}(r)}(v)$ be a face of dimension $\leq d-2$. Then $F\cup \{v\}\in \overline{\Delta}(r)$. Therefore, $\skel_{r}(F\cup \{v\})\subseteq \Delta$, so $\skel_{r-1}(F)\subseteq \lk_\Delta(v)$, and hence $F\in \overline{\lk_\Delta(v)}(r-1)$. On the other hand, if $F\in \lk_{\overline{\Delta}(r)}(v)$ is a $(d-1)$-face (thus, a cross-polytope), then by definition of $\overline{\Delta}(r)$ (see Definition \ref{Delta(j)}), there is $w\in V(\Delta)$ of the  same color as $v$ such that the $r$-skeleton of the cross-polytope on the vertex set $V(F)\cup\{v,w\}$ is contained in $\Delta$. Then $\skel_{r-1}(F)\subseteq \lk_\Delta(v)$, and so $F\in \overline{\lk_\Delta(v)}(r-1)$ in this case as well. This verifies eq.~(\ref{more_overline_links}).

Comparing equations (\ref{overline_link}) and (\ref{more_overline_links}), we conclude that $\X\supseteq  \overline{\Delta}(r)$, and the result follows. \hfill$\square$

\section{Balanced sphere bundles over $\mathbb{S}^1$}

The starting point of this section is motivated by the following question: in the class of all balanced triangulations of all closed $(d-1)$-manifolds that are not spheres, which complexes have the smallest number of vertices and how many vertices do they have? We show that in the class of homology manifolds that are not homology spheres the answer is $3d$, and that the same result holds for piecewise-linear (PL, for short) manifolds that are not PL spheres.

\begin{proposition}  \label{at_least_3d}
Let $\Delta$ be a $(d-1)$-dimensional balanced  $\field$-homology manifold without boundary that is not a $\field$-homology sphere.  Then $\Delta$ has at least three vertices of each color, and hence $f_0(\Delta) \geq 3d$.  The same result holds in the PL category.
\end{proposition}

\begin{proof} The statement is easy if $d=1$, so assume $d\geq 2$, and pick a color $j\in[d]$. Since $\Delta$ is pure, $\Delta_{[d]-j}$ is pure as well. Also, since $\Delta$ is a normal pseudomanifold, each facet of $\Delta_{[d]-j}$ is contained in two facets of $\Delta$. Thus there are at least two vertices of color $j$. Moreover, if there are exactly two vertices of color $j$, say, $u_1$ and $u_2$, it must be the case that $\lk_\Delta(u_1)=\Delta_{[d]-j}=\lk_\Delta(u_2)$. Therefore, in this latter case, $\Delta$ is the suspension of $\lk_{\Delta}(u_1)$.  Since $\lk_{\Delta}(u_1)$ is a $\field$-homology sphere and the suspension of a $\field$-homology sphere is also a $\field$-homology sphere, it follows that $\Delta$ is a $\field$-homology sphere.  Therefore $\Delta$ has at least three vertices of each color, which proves the desired result. The same proof applies verbatim for the PL category by replacing `homology sphere' with `PL sphere' throughout.
\end{proof}

Our goal now is, for each $d$, to exhibit a non-simply connected element of the balanced Walkup class $\BH^d$ (and hence a PL manifold that is not a sphere) with \textit{exactly} $3d$ vertices, see Theorem \ref{thm:3d-vertex-construction}. The underlying topological space of such a complex will be the $\mathbb{S}^{d-2}$-bundle over $\mathbb{S}^1$ that is orientable if $d$ is odd and nonorientable if $d$ is even. (Recall from \cite{Steenrod} that up 
to homeomorphism there are only two spherical bundles over $\mathbb{S}^1$ -- the trivial bundle, which is orientable, and the nonorientable bundle.)  Our construction can be considered a balanced analog of K\"uhnel's construction \cite{Kuhnel}. 

To start, suppose  $\widetilde{\Delta}$ is a balanced simplicial $(d-1)$-sphere with at least $3d$ vertices.  Suppose further there are two disjoint facets, $F = \{v_1,\ldots,v_d\}$ and $F' = \{v'_1,\ldots,v'_d\}$, in $\widetilde{\Delta}$ such that (1) $v_i$ and $v'_i$ have color $i$ for all $i$, and (2) $v_i$ and $v'_i$ do not have any common neighbors for all $i$.  Let $\Delta = \widetilde{\Delta}^{\phi}$ be the simplicial complex obtained from $\widetilde{\Delta}$ through a balanced handle addition that identifies $F$ with $F'$.  The following lemma will be useful in determining whether or not $\Delta$ is orientable.  Its first part is a variant of \cite[Proposition 11]{Joswig} and might be of interest in its own right.

\begin{lemma} \label{orientability_lemma} Let $\Delta$ be a balanced $(d-1)$-dimensional simplicial complex with $d\geq 3$. 
\begin{enumerate}  
\item  If $\Delta$ is a normal pseudomanifold, then $\Delta$ is orientable (i.e., $\widetilde{H}_{d-1}(\Delta; \Z)\cong \Z$) if and only if the facet-ridge graph of $\Delta$ is bipartite (i.e., $2$-colorable).
\item Assume that $\Delta$ is a connected, orientable simplicial manifold, and let $\Delta^\phi$ be obtained from $\Delta$ by a balanced handle addition that identifies facets $F$ and $F'$ of $\Delta$. Then $\Delta^\phi$ is orientable if and only if the distance from $F$ to $F'$ in the facet-ridge graph of $\Delta$ is odd.  
\end{enumerate}
\end{lemma}

\begin{proof} To verify Part 1, we fix a total order $\prec$ on the vertices of $\Delta$ that respects colors; in other words, the only restriction we place on $\prec$ is that if $\kappa(v)<\kappa(w)$, then $v\prec w$. 

Assume first that $\Delta$ is orientable, and let $\alpha=\sum \alpha_F [F]$ be a $(d-1)$-cycle that generates $\widetilde{H}_{d-1}(\Delta; \Z)$. Here the sum is over the facets of $\Delta$, $[F]$ denotes the ordered $d$-tuple of vertices of $F$ (so $[F]=[v_1,\ldots,v_d]$, where $\kappa(v_j)=j$), and $\alpha_F\in\{\pm 1\}$. Let $F'$ and $F''$ be two neighboring facets of $\Delta$, and let $G$ be their common ridge. Then $[d]-\kappa(G)$ consists of only one color, say, $j$. Since $\Delta$ is a pseudomanifold, $F$ and $F'$ are the only facets that contain $G$, and we obtain that the coefficient of $[G]$ in $\partial(\alpha)$ is equal to $(-1)^{j-1}\alpha_{F'}+(-1)^{j-1}\alpha_{F''}$. Thus for $\partial(\alpha)$ to be $0$, we must have 
\[
\alpha_{F'}=-\alpha_{F''} \quad \mbox{for every two neighboring facets } F', F''.
\]
Therefore, for any two facets, $F'$ and $F''$ of $\Delta$, the lengths of all paths between $F'$ and $F''$ in the facet-ridge graph of $\Delta$ must have the same parity (odd if $\alpha_{F'}=-\alpha_{F''}$ and even otherwise). Hence the facet-ridge graph of $\Delta$ is bipartite.

 Conversely, suppose the facet-ridge graph of $\Delta$ is bipartite, so we can refer to the facets of $\Delta$ as ``red" or ``blue" according to the chosen 2-coloring of this graph. Define $\alpha_F:=1$ if $F$ is red and $\alpha_F:=-1$ if $F$ is blue, and set $\alpha:=\sum \alpha_F[F]$. The same computation as in the previous paragraph then shows that $\partial(\alpha)=0$. Hence $\alpha$ is a non-zero element of $\widetilde{H}_{d-1}(\Delta; \Z)$, and so $\Delta$ is orientable.

For Part 2, since $\Delta$ is orientable, we can assume (by Part 1) that the facet ridge graph of $\Delta$ is endowed with a proper 2-coloring. Let $F_1,\ldots, F_d$ be the neighboring facets of $F$ and let $F'_1,\ldots, F'_d$ be the neighboring facets of $F'$, where the indexing is chosen so that $\kappa(F\cap F_i)=\kappa(F'\cap F'_i)=[d]-\{i\}$. Then the facet-ridge graph of $\Delta^\phi$ is obtained from the facet-ridge graph of $\Delta$ by removing the two vertices corresponding to $F$ and $F'$, and connecting the vertices corresponding to $F_i$ and $F'_i$ by an edge for each $i\in[d]$. 

If the distance from $F$ to $F'$ in the facet-ridge graph of $\Delta$ is odd, then $F$ and $F'$ have opposite colors. Consequently, $F_i$ and $F'_i$ have opposite colors (for all $i$), and hence adding an edge between the vertices corresponding to $F_i$ and $F'_i$ (for each $i\in[d]$) does not affect the 2-colorability of the graph. Thus, in this case, the facet-ridge graph of $\Delta^\phi$ is also bipartite, and hence $\Delta^\phi$ is orientable.

On the other hand, if the distance from $F$ to $F'$ in the facet-ridge graph of $\Delta$ is even, then the distance from $F_1$ to $F'_1$ is also even, so every path from $F_1$ to $F'_1$ has an even length in this graph. Let $F_1=H_0, H_1,\ldots, H_{2k-1}, H_{2k}=F'_1$ be such a path that passes neither through $F$ nor through $F'$ (it exists since $d\geq 3$ and the facet-ridge graph of $\Delta$ is $d$-connected, see \cite{Barnette-connectivity}). Then the same path together with an added edge between $F_1$ and $F'_1$ forms an odd cycle in the facet-ridge graph of $\Delta^\phi$. Thus, the facet-ridge graph of $\Delta^\phi$ is not bipartite, and so $\Delta^\phi$ is nonorientable.
\end{proof}

We are now in a position to describe our $3d$-vertex construction.

\begin{theorem} \label{thm:3d-vertex-construction} There exists a balanced simplicial manifold $\BM_{d}$ with $3d$ vertices that triangulates $\mathbb{S}^{d-2} \times \mathbb{S}^1$ if $d$ is odd, and the nonorientable $\mathbb{S}^{d-2}$-bundle over $\mathbb{S}^1$ if $d$ is even.
\end{theorem}

\begin{proof} Let $\Delta_1$, $\Delta_2$, and $\Delta_3$ be boundary complexes of  $d$-dimensional cross-polytopes with $V(\Delta_1)=\{x_1,\ldots,x_d\} \cup \{y_1,\ldots,y_d\}$, $V(\Delta_2)=\{y'_1,\ldots,y'_d\} \cup \{z_1,\ldots,z_d\}$, and $V(\Delta_3)=\{z'_1,\ldots,z'_d\} \cup \{x'_1,\ldots,x'_d\}$, where we assume that each vertex with index $i$ has color $i$. Now let $\widetilde{\Delta} = \Delta_1 \# \Delta_2 \# \Delta_3$, where the first connected sum identifies $y_i$ with $y'_i$ for all $i$, and the second connected sum identifies $z_i$ with $z'_i$ for all $i$.  

Observe that in $\widetilde{\Delta}$, any vertex $x_i$ is only adjacent to vertices in $\{x_1,\ldots,x_d\}\cup\{y_1, \ldots, y_d\}$, and any vertex $x'_i$ is only adjacent to vertices in $\{z_1,\ldots,z_d\}\cup\{x'_1,\ldots,x'_d\}$.  Thus we can perform a balanced handle addition to $\widetilde{\Delta}$ by identifying $x_i$ to $x'_i$ for all $i$ and removing the identified facets $\{x_1,\ldots,x_d\}$ and $\{x'_1,\ldots,x'_d\}$.  Let $\BM_{d}$ denote the resulting simplicial complex. After all identifications, we can write $V(\BM_{d}) = \{x_1,\ldots,x_d,y_1,\ldots,y_d,z_1,\ldots,z_d\}$.  

Since $\widetilde{\Delta}$ is a sphere and $\BM_{d}$ is obtained from $\widetilde{\Delta}$ by handle addition, the resulting space is an $\mathbb{S}^{d-2}$-bundle over $\mathbb{S}^1$.  Finally, since the distance from the facet $\{x_1,\ldots,x_d\}$ to the facet $\{x'_1,\ldots,x'_d\}$ in the facet-ridge graph of $\widetilde{\Delta}$ is $3(d-1)+1=3d-2$, Part 2 of Lemma \ref{orientability_lemma} implies that for $d\geq 3$, $\BM_{d}$ is orientable if $d$ is odd, and it is nonorientable if $d$ is even. For $d=2$, the resulting complex is homeomorphic to $\mathbb{S}^1$, and hence it is the non-orientable $\mathbb{S}^0$-bundle over $\mathbb{S}^1$.
\end{proof}

We now show that $\BM_d$ possesses balanced analogs of virtually all properties that K\"uhnel's construction \cite{Kuhnel} has. We start with the following observation.

\begin{remark} \label{BM_d-prop}
Since there is a natural bijection between the set of facets of $\C^*_d$ on the vertex set $\{x_1,\ldots, x_d,y_1,\ldots,y_d\}$ and the set of words of length $d$ in the alphabet $\{x,y\}$ (e.g., under this bijection the facet $\{x_1,y_2,\ldots, y_d\}$ corresponds to the word $xyy\cdots y$), the construction of $\BM_{d}$ implies that there is a natural bijection between the set of facets of $\BM_{d}$ and the set of words of length $d$ in the alphabet $\{x,y, z\}$ containing exactly two of the letters. It then follows that (i) the group $S_3\times S_d$ acts \textit{vertex transitively} on $\BM_{d}$, where $S_3$ acts by permuting the letters of the alphabet and $S_d$ by permuting the positions of the $d$ letters inside each word; and (ii) for $d\geq 3$, the graph of $\BM_{d}$ is the graph $K_{3,3,\ldots,3}$ --- the complete $d$-partite graph (also known as Tur\'an graph) on $3+3+\cdots +3$ vertices. In other words,  $\BM_{d}$ is a ``\textit{balanced 2-neighborly}" simplicial manifold.
\end{remark}

\begin{proposition} \label{isom-to-BM_d} For all $d\geq 3$, the complex $\BM_d$ is (up to simplicial isomorphism) the only $3d$-vertex element of $\BH^d$ that is not a sphere.
\end{proposition} 
\begin{proof} Let $\Delta$ be a $3d$-vertex element of $\BH^d$ that is not a sphere, and let $k\geq 1$ denote the number of handle additions used in the construction of $\Delta$. Then $\overline{g}_2(\Delta)=4\binom{d}{2}k\geq 4\binom{d}{2}=\overline{g}_2(\BM_d)$ and $f_0(\Delta)=f_0(\BM_d)$, and so $f_1(\Delta)\geq f_1(\BM_d)$. On the other hand, the graph of $\BM_d$ is $K_{3,3,\ldots,3}$ --- the complete $d$-partite graph with 3 vertices of each color, and hence it contains the graph of $\Delta$. We conclude that $f_1(\Delta)=f_1(\BM_d)$ and that $k=1$. Therefore, for $\Delta$ to have $3d$ vertices, it must be obtained from three (boundary complexes of) cross-polytopes by using two balanced connected sums and one balanced handle addition, i.e., $\Delta=(\Gamma_1\#_{\psi_1}\Gamma_2\#_{\psi_2}\Gamma_3)^\phi$, where $\Gamma_1$, $\Gamma_2$, and $\Gamma_3$ are all cross-polytopes. 

Consider $\Lambda=\Gamma_1\#_{\psi_1}\Gamma_2\#_{\psi_2}\Gamma_3$. Observe that no vertex of $\Lambda$ belongs to all three cross-polytopes: indeed, if $v_i$ were such a vertex of $\Lambda$ of color $i$, then for any color $j\neq i$, every two vertices of $\Lambda$ of color $j$ would have $v_i$ as their common neighbor, and so there would be no way to add a balanced handle to $\Lambda$. Thus, without loss of generality, we can assume that $V(\Gamma_1)=\{x_1,\ldots, x_d,y_1\ldots,y_d\}$, $V(\Gamma_2)=\{y'_1,\ldots,y'_d, z_1,\ldots, z_d\}$, and $V(\Gamma_3)=\{z'_1,\ldots,z'_d,x'_1,\ldots,x'_d\}$, where $\psi_1$ identifies $y_i$ with $y'_i$, and $\psi_2$ identifies $z_i$ with $z'_i$ for $i=1,\ldots,d$. But then, for each $i\in[d]$,  the only two vertices of $\Lambda$ of color $i$ that do not have a common neighbor are $x_i$ and $x'_i$. Therefore, the only balanced handle that can be added to $\Lambda$ is the one that identifies $x_i$ with $x'_i$ for all $i$, and hence $\Delta=\BM_d$.
\end{proof}

We close our discussion of the properties of $\BM_d$ by briefly talking about the face numbers. We have: $f_0(\BM_{d})=3d$ and $\overline{g}_2(\BM_{d})=4\binom{d}{2}$. Hence, it follows from Proposition \ref{at_least_3d} and Theorem \ref{t-covering} that for $d\geq 3$, in the class of all balanced $(d-1)$-dimensional homology manifolds with non-trivial $\widetilde{H}_1(-,\Q)$, the complex $\BM_{d}$ simultaneously minimizes $f_0$ and $f_1$. In fact, it minimizes the entire $f$-vector: 

\begin{theorem}  \label{BM_d-face_numbers}
Let $\Delta$ be a balanced $(d-1)$-dimensional homology manifold with $\beta_1(\Delta;\Q)\neq 0$.
\begin{enumerate}
\item If $d\geq 2$, then $f_{i-1}(\Delta)\geq f_{i-1}(\BM_{d})$ for all $0<i\leq d$.
\item Moreover, if $d\geq 5$, and $(f_0(\Delta), f_1(\Delta), f_2(\Delta))=(f_0(\BM_{d}), f_1(\BM_{d}),f_2(\BM_{d}))$, then $\Delta$ is isomorphic to $\BM_d$.
\end{enumerate}
\end{theorem}

\begin{proof} Part 1 clearly holds for $d=2$, since any 1-dimensional manifold satisfies $f_0=f_1$. It also holds for $d=3$, since in this case $2f_1=3f_2$. For $d\geq 4$, the proof of Part 1 follows by summing over vertex links and is almost identical to the proof of its non-balanced counterpart proved in \cite[Theorem 4.7]{Swartz}.  We apply Theorem \ref{balancedLBT} (its $f$-vector form as discussed in \cite[Theorem 5.3]{Goff-Klee-Novik}) instead of the classic LBT and the inequality $\overline{g}_2\geq 4\binom{d}{2}$ of Theorem \ref{t-covering} in place of the inequality $g_2\geq\binom{d+1}{2}$. We omit the details.

For Part 2, note that if $(f_0(\Delta), f_1(\Delta), f_2(\Delta))=(f_0(\BM_{d}), f_1(\BM_{d}),f_2(\BM_{d}))$, then $h_1(\Delta)=h_1(\BM_d)$ and $h_3(\Delta)=h_3(\BM_d)$, and so
\[ \frac{h_3(\Delta)}{\binom{d}{3}}-\frac{h_1(\Delta)}{\binom{d}{1}}= \frac{h_3(\BM_d)}{\binom{d}{3}}-\frac{h_1(BM_d)}{\binom{d}{1}}=0,
\]
where the last step follows from Remark \ref{h_3-vs-h_1}. Applying Remark \ref{h_3-vs-h_1} once again, we conclude that for every vertex $v$ of $\Delta$, $\overline{g}_2(\lk_\Delta(v))=0$. Since $1+\dim\lk_\Delta(v)=d-1\geq 4$, Theorem \ref{balancedLBT-equality} yields that all vertex links of $\Delta$ are stacked cross-polytopal spheres, and so by Corollary \ref{balanced8.4}, the complex $\Delta$ lies in the balanced Walkup class $\BH^d$. Proposition \ref{isom-to-BM_d} then guarantees that $\Delta$ is isomorphic to $\BM_d$.
\end{proof} 

In the remainder of this section we concentrate on the following question: for which values of $n\geq 3d$ is there a balanced $n$-vertex triangulation of the orientable (resp.~nonorientable)  $\mathbb{S}^{d-2}$-bundle over $\mathbb{S}^{1}$? While the situation is obvious for $d=2$, we offer the following theorem and conjecture for $d\geq 3$. (For similar results in the non-balanced case, see \cite{Bagchi-Datta-bundles} and \cite{Chestnut-Sapir-Swartz}.)

\begin{theorem} \label{Delta_k,d}  For all $d \geq 3$ and $k\geq 2$ there exists a balanced simplicial manifold $\Delta_{k,d}$ with $3d+k$ vertices  that triangulates $\mathbb{S}^{d-2} \times \mathbb{S}^1$, and a balanced simplicial manifold $\widetilde{\Delta}_{k,d}$ with $3d+k$ vertices that triangulates the nonorientable $\mathbb{S}^{d-2}$-bundle over $\mathbb{S}^1$.
\end{theorem}

\begin{proof}
The proof is very similar to the proof of Theorem \ref{thm:3d-vertex-construction}. We first consider the case of an even $k$. Let $\Sigma_1$ denote the boundary complex of a $(d-2)$-dimensional cross-polytope whose vertices are labeled as $\{x_1,x_2,\ldots,x_{d-2}\} \cup \{y_1,y_2,\ldots,y_{d-2}\}$ where $\kappa(x_i) = \kappa(y_i) = i$, and let $C$ denote the cycle graph on $4+k$ vertices, which are labeled sequentially as $\{u_1,\ldots, u_k,x_{d-1},x_d,y_{d-1},y_d\}$.  Then $C$ is balanced if we declare that $\kappa(u_{2i-1})=\kappa(x_{d-1}) = \kappa(y_{d-1}) = d-1$ and $\kappa(u_{2i}) = \kappa(x_d) = \kappa(y_d) = d$.  Let $\Delta_1 = \Sigma_1*C$.  Alternatively, we may view $\Delta_1$ as the $(d-2)$-fold suspension of $C$.  

As in the proof of Theorem \ref{thm:3d-vertex-construction}, let $\Delta_2$ and $\Delta_3$ be the boundary complexes of $d$-dimensional cross-polytopes with $V(\Delta_2)=\{y'_1,\ldots,y'_d\} \cup \{z_1,\ldots,z_d\}$ and $V(\Delta_3)=\{z'_1,\ldots,z'_d\} \cup \{x'_1,\ldots,x'_d\}$.  Let $\widetilde{\Delta} = \Delta_1 \# \Delta_2 \# \Delta_3$, where the first connected sum identifies vertex $y_i$ with vertex $y'_i$ for all $i$, and the second connected sum identifies vertex $z_i$ with vertex $z'_i$ for all $i$.  

 Consider the following three facets of $\widetilde{\Delta}$: $X' = \{x'_1,\ldots, x'_{d-1},x'_d\}$, $X = \{x_1,\ldots,x_{d-1},x_d\}$, and $\widetilde{X} = \{x_1,\ldots,x_{d-1},u_k\}$.  Once again, the vertices in $X'$ do not have any neighbors in common with any of the vertices in either $X$ or $\widetilde{X}$.  Therefore, we can perform two separate balanced handle additions.  Let $\Gamma$ be the simplicial complex obtained from $\widetilde{\Delta}$ through a balanced handle addition that identifies the vertices of $X'$ to the vertices of $X$, and let $\Gamma'$ be obtained from $\widetilde{\Delta}$ through a balanced handle addition that identifies $X'$ to $\widetilde{X}$.  Both $\Gamma$ and $\Gamma'$ are $\mathbb{S}^{d-2}$-bundles over $\mathbb{S}^1$ with $3d+k$ vertices.  Furthermore, by Part 2 of Lemma \ref{orientability_lemma}, one of them is orientable and the other one is not. This establishes the existence of both $\Delta_{k,d}$ and $\widetilde{\Delta}_{k,d}$ for every even $k\geq 2$. 

Finally, if $k\geq 3$ is odd, then let $\Delta_1$ be the $(d-3)$-fold suspension of $\C^*_3\# \C^*_3$, let $\Delta_2$ be the $(d-2)$-fold suspension over the cycle of length $k+1$, and let $\Delta_3$ be $\C^*_d$. Then $\Delta_1$ has $2\cdot(d-3)+9$ vertices, $\Delta_2$ has $2\cdot(d-2)+(k+1)$ vertices, and $\Delta_3$ has $2d$ vertices, so the resulting balanced connected sum $\Delta_1\#\Delta_2\#\Delta_3$ has $4d+k$ vertices.  The rest of the construction follows similarly to the case that $k$ is even by performing a balanced handle operation to $\Delta_1\#\Delta_2\#\Delta_3$, and we omit the details for the sake of brevity.
\end{proof}

\begin{conjecture} \label{no_triangulation} For all $d\geq 3$, the following hold.
\begin{enumerate}
\item The orientable $\mathbb{S}^{d-2}$-bundle over $\mathbb{S}^{1}$ has no balanced $3d$-vertex triangulation if $d$ is even; the nonorientable $\mathbb{S}^{d-2}$-bundle over $\mathbb{S}^{1}$ has no balanced $3d$-vertex triangulation if $d$ is odd.
\item Neither the orientable nor the nonorientable $\mathbb{S}^{d-2}$-bundle over $\mathbb{S}^{1}$ has a balanced $(3d+1)$-vertex triangulation.
\end{enumerate}
\end{conjecture}

\noindent Note that according to Proposition \ref{isom-to-BM_d}, if there is a balanced $3d$-vertex triangulation of $\mathbb{S}^{d-2}\times \mathbb{S}^1$ for even $d\geq 3$ (resp.~of the non-orientable bundle for odd $d\geq 3$), then such a triangulation does not belong to the balanced Walkup class.

We close this section with collecting some evidence in favor of Conjecture \ref{no_triangulation}. For small values of $d$, Part 1 of the conjecture follows from the following strengthening of Proposition \ref{isom-to-BM_d}. Recall that $\widetilde{\chi}(\Delta):=\sum_{i=0}^{d-1}(-1)^i \beta_i(\Delta;\field)$ denotes the \textit{reduced Euler characteristic} of $\Delta$. Recall also, that for a homology $(d-1)$-manifold $\Delta$, the following Dehn-Sommerville relations hold \cite{Klee-DS}:
\[
h_{d-i}(\Delta) =h_i(\Delta)+(-1)^i \binom{d}{i}\left(\widetilde{\chi}(\Delta)+(-1)^d\right).
\]

\begin{proposition} Let $d\in\{3,4,5\}$. If $\Delta$ is a balanced $3d$-vertex triangulation of a homology $(d-1)$-manifold with $\beta_1(\Delta; \Q)\neq 0$ and $\widetilde{\chi}(\Delta)=\widetilde{\chi}(\BM_d)$, then $\Delta$ is isomorphic to $\BM_d$. Thus, the first part of Conjecture \ref{no_triangulation} holds for all $d\leq 5$.
\end{proposition}

\begin{proof} We are given that $f_0(\Delta)=3d=f_0(\BM_d)$, and we also know from Part 1 of Theorem~\ref{BM_d-face_numbers} that  $f_1(\Delta)\geq f_1(\BM_d)$. However, by Remark \ref{BM_d-prop}, the graph of $\BM_d$ is the complete $d$-partite graph with 3 vertices of each color. It thus follows that $\Delta$ and $\BM_d$ have the same graph, and hence that $(h_0(\Delta),h_1(\Delta), h_2(\Delta))=(h_0(\BM_d),h_1(\BM_d), h_2(\BM_d))$. Since $d\leq 5$ and since, by the Dehn-Sommerville relations, the Euler characteristic and the first half of the $h$-vector of any homology manifold  determine the entire $h$-vector, we infer that $f(\Delta)=f(\BM_d)$. The case $d=5$ of the statement then follows from Part 2 of Theorem \ref{BM_d-face_numbers}.

We now consider the case of $d=3$. (Recall that homology manifolds of dimension $\leq 3$ are simplicial manifolds.) Label the vertices of $\Delta$ so that $V_1(\Delta)=\{x_1,y_1,z_1\}$. Since the graph of $\Delta$ is $K_{3,3,3}$ and since $\Delta$ is a manifold, the link of each vertex in $\Delta$ is a graph cycle on 6 vertices. In other words, the complement of $\lk_\Delta(x_1)$ in $\Delta_{\{2,3\}}= K_{3,3}$ consists of 3 edges that form a perfect matching $M_x$ of $K_{3,3}$, and the same holds for the links of $y_1$ and $z_1$. Moreover, since every edge of $\Delta$ is in exactly two facets, the three perfect matchings $M_x$, $M_y$, and $M_z$ are pairwise disjoint and their union is $K_{3,3}$. Hence there exist $e_x\in M_x$, $e_y\in M_y$, and $e_z\in M_z$ such that $\{e_x, e_y, e_z\}$ is also a matching of $K_{3,3}$. The sets $F_x:=e_x\cup\{x_1\}$, $F_y:=e_y\cup\{y_1\}$, and $F_z:=e_z\cup\{z_1\}$ then form a collection of three pairwise disjoint missing 2-faces of $\Delta$. As in the proof of Lemma \ref{missing(d-1)face}, cut along $F_x$ and patch with two 2-simplices. Then do the same with $F_y$ and $F_z$. Since $\overline{g}_2(\Delta)=\overline{g}_2(BM_3)=4\binom{3}{2}$, at most one of these three operations is ``undoing a handle", and we conclude (by counting vertices) that the complex we end up with is a disjoint union of three copies of $\C^*_3$. Therefore, $\Delta\in \BH^3$, and so $\Delta$ is isomorphic to $\BM_3$ by Proposition \ref{isom-to-BM_d}.

To treat the case of $d=4$, we first show that (up to isomorphism) the only balanced 2-sphere on $3+3+3$ vertices is $\C^*_3\#\C^*_3$. Indeed, if $\Gamma$ is such a sphere with $f_0(\Gamma)=9$, then  $f_1(\Gamma)=3\cdot 9-6=21$ and $f_2=14$. Hence w.l.o.g.~$f_{\{2,3\}}(\Gamma)\geq \frac{21}{3}=7$. On the other hand, since $14=f_2(\Gamma)= f_1(\lk_\Gamma(x_1))+ f_1(\lk_\Gamma(y_1))+f_1(\lk_\Gamma(z_1))$, and since all vertex links are even cycles, there must be a vertex of $\Gamma$ of color 1, say $x_1$,  whose link is a hexagon. Therefore, there is an edge $e\in\Gamma_{\{2,3\}}$ that is not in the link of $x_1$ although both of its endpoints are. Thus $F:=e\cup\{x_1\}$ is a missing 2-face of $\Gamma$. The claim follows by cutting $\Gamma$ along $F$ and patching with two 2-simplices.

Now let $\Delta$ be a 3-dimensional complex as in the statement of the proposition, and let $v\in V(\Delta)$. Since the graph of $\Delta$ is $K_{3,3,3,3}$, the link of $v$ is a balanced 2-sphere on $3+3+3$ vertices, and so, $\lk_\Delta(v)=C_1\#C_2$, where $C_i$ (for $i=1,2$) is isomorphic to $\C^*_3$. In particular, $\lk_\Delta(v)$ has a unique missing 2-face; we denote it by $F$. 

{\bf Case 1:} $F$ is not a face of $\Delta$. In this case we use the same trick as in \cite[Proposition 2.3]{Swartz-finiteness}: introduce two new vertices $v_1$ and $v_2$ of the same color as $v$, remove the ball $B:=v*\lk_\Delta(v)$ from $\Delta$ and replace it with the ball $B':=(v_1*(C_1\cup \{F\}) \cup (v_2*(C_2\cup \{F\})$. The resulting complex $\Delta'$ is balanced. Further, since $B$ and $B'$ have the same boundary complex (indeed, $\partial(B)=\partial(B')=\lk_\Delta(v)$), $\Delta'$ is also a simplicial manifold with $\|\Delta\|\cong\|\Delta'\|$. On the other hand, $f_0(\Delta')=f_0(\Delta)+1$ and $f_1(\Delta')=f_1(\Delta)+3$, yielding that $\overline{g}_2(\Delta')=\overline{g}_2(\Delta)-3=4\binom{4}{3}-3$. This however contradicts the assertion of Theorem \ref{t-covering}, and so this case is impossible.

{\bf Case 2:}  $F$ is a face of $\Delta$. In this case, $F':=F\cup\{v\}$ is a missing 3-face of $\Delta$. Cutting $\Delta$ along $F'$ and patching with two 3-simplices results in a simplicial manifold, $\Lambda$. Note that $\Lambda$ is connected (this is because the graph of $\Delta$ is $K_{3,3,3,3}$). Thus $\Delta$ is obtained from $\Gamma$ through a balanced handle addition, and $\overline{g}_2(\Lambda)=\overline{g}_2(\Delta)-4\binom{d}{2}=0$. Hence, by Theorem \ref{balancedLBT-equality}, $\Lambda$ is a stacked cross-polytopal sphere. Therefore, $\Delta$ is in $\BH^4$, and so $\Delta$ is isomorphic to $\BM_4$ by Proposition \ref{isom-to-BM_d}. 
\end{proof}

As for the second part of Conjecture \ref{no_triangulation}, we have:

\begin{proposition} The second part of Conjecture \ref{no_triangulation} holds for $d\in\{3,4\}$.
\end{proposition}
\begin{proof} Assume to the contrary that $\Delta$ is a balanced $(3d+1)$-vertex triangulation of one of the $\mathbb{S}^{d-2}$-bundles over $\mathbb{S}^{1}$, and w.l.o.g.~assume that $|V_d(\Delta)|=4$, while $|V_i(\Delta)|=3$ for all $i\in[d-1]$. 

We start with the case of $d=3$. In this case, the link of each vertex must be an even cycle, and hence it must contain at most 6 vertices. On the other hand, since $\widetilde{\chi}(\Delta)=0-2+1=-1$ and $f_0(\Delta)=10$, it follows (e.g., from the Dehn-Sommerville relations) that $f_1(\Delta)=30$, and $f_2(\Delta)=20$. Thus the average degree of a vertex is $2f_1/f_0=60/10=6$, and we conclude that all vertex links are hexagons. However, since every 2-face contains exactly one vertex of color 3, it then follows that $f_2(\Delta)=4\times 6=24\neq 20$, which is a contradiction.

Assume now that $d=4$. Then $\Delta$ is an odd-dimensional manifold, and hence it is an Eulerian complex (i.e., for every face $F$ of $\Delta$, including the empty face, $\widetilde{\chi}(\lk_\Delta(F))=\widetilde{\chi}(\mathbb{S}^{d-1-|F|})$). Since $\beta_1(\Delta;\Q)\neq 0$, Theorem \ref{t-covering} implies that 
\[
2h_2(\Delta)\geq 3h_1(\Delta)+4\binom{4}{2}=3(13-4)+24=51, \quad \mbox{and so} \quad h_2(\Delta)\geq 26.
\]
On the other hand, 
\begin{equation}  \label{26-vs-24}
26\leq h_2(\Delta)=\sum_{\substack{S \subseteq [4]\\ |S| = 2}}h_S(\Delta)=\sum_{\substack{S \subseteq [3]\\ |S| = 2}}\left(h_S(\Delta)+h_{[4]-S}(\Delta)\right)=2\sum_{\substack{S \subseteq [3]\\ |S| = 2}}h_S(\Delta),\end{equation} where the last step is by the generalized Dehn-Sommerville relations \cite{BayerBillera}. Finally, if $S$ is any 2-element subset of $[3]$, then $h_S(\Delta)=f_1(\Delta_S)-f_0(\Delta_S)+1\leq 3\cdot 3-6+1=4$, and so the right-hand side of eq.~(\ref{26-vs-24}) is at most 24, which is a contradiction.
\end{proof} 

It would be very interesting to prove (or disprove) Conjecture \ref{no_triangulation} for all values of $d$. Even more intriguing question is to characterize all possible pairs $(f_0,f_1)$ of vertices and edges of any \textit{balanced} triangulation of an $\mathbb{S}^{d-2}$-bundle over $\mathbb{S}^{1}$. In the case of \textit{all} triangulations of such a space, this was done in \cite{Chestnut-Sapir-Swartz}.

\begin{remark}
In the time that this paper was under review, Zheng \cite{Zheng} proved the remaining cases of Conjecture \ref{no_triangulation}.
\end{remark}

%

\section*{Acknowledgments} We would like to thank Ed Swartz for helpful conversations and the referees for helpful suggestions.


{\small
\bibliography{balancedbib}

\begin{thebibliography}{10}

\bibitem{AsimowRothI}
L.~Asimow and B.~Roth.
\newblock The rigidity of graphs.
\newblock {\em Trans. Amer. Math. Soc.}, 245:279--289, 1978.

\bibitem{AsimowRothII}
L.~Asimow and B.~Roth.
\newblock The rigidity of graphs. {II}.
\newblock {\em J. Math. Anal. Appl.}, 68(1):171--190, 1979.

\bibitem{Bagchi}
B.~Bagchi.
\newblock {The mu vector, Morse inequalities and a generalized lower bound
  theorem for locally tame combinatorial manifolds}.
\newblock arXiv:1405.5675.

\bibitem{BagchiDatta-LBT}
B.~Bagchi and B.~Datta.
\newblock Lower bound theorem for normal pseudomanifolds.
\newblock {\em Expo. Math.}, 26(4):327--351, 2008.

\bibitem{Bagchi-Datta-bundles}
B.~Bagchi and B.~Datta.
\newblock Minimal triangulations of sphere bundles over the circle.
\newblock {\em J. Combin. Theory Ser. A}, 115(5):737--752, 2008.

\bibitem{Bagchi-Datta-stellated}
B.~Bagchi and B.~Datta.
\newblock On {$k$}-stellated and {$k$}-stacked spheres.
\newblock {\em Discrete Math.}, 313(20):2318--2329, 2013.

\bibitem{Barnette-LBT-pseudomanifolds}
D.~Barnette.
\newblock Graph theorems for manifolds.
\newblock {\em Israel J. Math.}, 16:62--72, 1973.

\bibitem{Barnette-LBT}
D.~Barnette.
\newblock A proof of the lower bound conjecture for convex polytopes.
\newblock {\em Pacific J. Math.}, 46:349--354, 1973.

\bibitem{Barnette-connectivity}
D.~Barnette.
\newblock Decompositions of homology manifolds and their graphs.
\newblock {\em Israel J. Math.}, 41(3):203--212, 1982.

\bibitem{BayerBillera}
M.~M. Bayer and L.~J. Billera.
\newblock Generalized {D}ehn-{S}ommerville relations for polytopes, spheres and
  {E}ulerian partially ordered sets.
\newblock {\em Invent. Math.}, 79(1):143--157, 1985.

\bibitem{Billera-Lee}
L.~Billera and C.~W. Lee.
\newblock A proof of the sufficiency of {M}c{M}ullen's conditions for
  {$f$}-vectors of simplicial convex polytopes.
\newblock {\em J. Combin. Theory Ser. A}, 31(3):237--255, 1981.

\bibitem{Bjorner-Frankl-Stanley}
A.~Bj{\"o}rner, P.~Frankl, and R.~Stanley.
\newblock The number of faces of balanced {C}ohen-{M}acaulay complexes and a
  generalized {M}acaulay theorem.
\newblock {\em Combinatorica}, 7(1):23--34, 1987.

\bibitem{Browder-Klee}
J.~Browder and S.~Klee.
\newblock Lower bounds for {B}uchsbaum* complexes.
\newblock {\em European J. Combin.}, 32:146--153, 2011.

\bibitem{Chestnut-Sapir-Swartz}
J.~Chestnut, J.~Sapir, and E.~Swartz.
\newblock Enumerative properties of triangulations of spherical bundles over
  {$S^1$}.
\newblock {\em European J. Combin.}, 29(3):662--671, 2008.

\bibitem{Datta-Murai}
B.~Datta and S.~Murai.
\newblock On stacked triangulated manifolds.
\newblock arXiv:1407.6767.

\bibitem{Ehrenborg-Karu-cd}
R.~Ehrenborg and K.~Karu.
\newblock Decomposition theorem for the {${\bf cd}$}-index of {G}orenstein
  posets.
\newblock {\em J. Algebraic Combin.}, 26(2):225--251, 2007.

\bibitem{Fogelsanger}
A.~Fogelsanger.
\newblock {\em The generic rigidity of minimal cycles}.
\newblock ProQuest LLC, Ann Arbor, MI, 1988.
\newblock Thesis (Ph.D.)--Cornell University.

\bibitem{Frankl-Furedi-Kalai}
P.~Frankl, Z.~F{\"u}redi, and G.~Kalai.
\newblock Shadows of colored complexes.
\newblock {\em Math. Scand.}, 63(2):169--178, 1988.

\bibitem{Goff-Klee-Novik}
M.~Goff, S.~Klee, and I.~Novik.
\newblock Balanced complexes and complexes without large missing faces.
\newblock {\em Ark. Mat.}, 46(2):335--350, 2011.

\bibitem{Joswig}
M.~Joswig.
\newblock Projectivities in simplicial complexes and colorings of simple
  polytopes.
\newblock {\em Math. Z.}, 240(2):243--259, 2002.

\bibitem{Kalai-rigidity}
G.~Kalai.
\newblock Rigidity and the lower bound theorem. {I}.
\newblock {\em Invent. Math.}, 88(1):125--151, 1987.

\bibitem{Karu-cd}
K.~Karu.
\newblock The {$cd$}-index of fans and posets.
\newblock {\em Compos. Math.}, 142(3):701--718, 2006.

\bibitem{Klee-DS}
V.~Klee.
\newblock A combinatorial analogue of {P}oincar\'e's duality theorem.
\newblock {\em Canad. J. Math.}, 16:517--531, 1964.

\bibitem{Kubitzke-Murai}
M.~Kubitzke and S.~Murai.
\newblock Balanced generalized lower bound inequality for simplicial polytopes.
\newblock \texttt{http://arxiv.org/abs/1503.06430}, 2015.

\bibitem{Kuhnel}
W.~K{\"u}hnel.
\newblock Higher-dimensional analogues of {C}z\'asz\'ar's torus.
\newblock {\em Results Math.}, 9(1-2):95--106, 1986.

\bibitem{Lutz-page}
F.~Lutz.
\newblock The manifold page.
\newblock \texttt{http://page.math.tu-berlin.de/\textasciitilde lutz/stellar/}.

\bibitem{Lutz-thesis}
F.~Lutz.
\newblock {\em Triangulated manifolds with few vertices and vertex-transitive
  group actions}.
\newblock Berichte aus der Mathematik. [Reports from Mathematics]. Verlag
  Shaker, Aachen, 1999.
\newblock Dissertation, Technischen Universit{\"a}t Berlin, Berlin, 1999.

\bibitem{Lutz-Sulanke-Swartz}
F.~Lutz, T.~Sulanke, and E.~Swartz.
\newblock {$f$}-vectors of 3-manifolds.
\newblock {\em Electron. J. Combin.}, 16(2, Special volume in honor of Anders
  Bjorner):Research Paper 13, 33, 2009.

\bibitem{McMullen04}
P.~McMullen.
\newblock Triangulations of simplicial polytopes.
\newblock {\em Beitr\"age Algebra Geom.}, 45(1):37--46, 2004.

\bibitem{McMullen-Walkup}
P.~McMullen and D.~W. Walkup.
\newblock A generalized lower-bound conjecture for simplicial polytopes.
\newblock {\em Mathematika}, 18:264--273, 1971.

\bibitem{Munkres}
J.~Munkres.
\newblock Topological results in combinatorics.
\newblock {\em Michigan Math. J.}, 31(1):113--128, 1984.

\bibitem{Murai-Nevo}
S.~Murai and E.~Nevo.
\newblock On the generalized lower bound conjecture for polytopes and spheres.
\newblock {\em Acta. {M}ath.}, 210:185--202, 2013.

\bibitem{Murai-Nevo2}
S.~Murai and E.~Nevo.
\newblock On {$r$}-stacked triangulated manifolds.
\newblock {\em J. Algebraic Combin.}, 39(2):373--388, 2014.

\bibitem{Murai-15}
Satoshi Murai.
\newblock Tight combinatorial manifolds and graded betti numbers.
\newblock {\em Collect. Math.}, 2015.
\newblock
  \texttt{http://link.springer.com/article/10.1007\%2Fs13348-015-0137-z}.

\bibitem{Novik-Swartz-buchsbaum}
I.~Novik and E.~Swartz.
\newblock Socles of {B}uchsbaum modules, complexes and posets.
\newblock {\em Adv. Math.}, 222(6):2059--2084, 2009.

\bibitem{Stanley-balanced}
R.P. Stanley.
\newblock Balanced {C}ohen-{M}acaulay complexes.
\newblock {\em Trans. Amer. Math. Soc.}, 249(1):139--157, 1979.

\bibitem{Stanley-gthm}
R.P. Stanley.
\newblock The number of faces of a simplicial convex polytope.
\newblock {\em Adv. in Math.}, 35(3):236--238, 1980.

\bibitem{Stanley-cd}
R.P. Stanley.
\newblock Flag {$f$}-vectors and the {$cd$}-index.
\newblock {\em Math. Z.}, 216(3):483--499, 1994.

\bibitem{Stanley-green-book}
R.P. Stanley.
\newblock {\em Combinatorics and Commutative Algebra}.
\newblock Birkh{\"a}user, Boston, 2. edition, 1996.

\bibitem{Steenrod}
N.E. Steenrod.
\newblock The classification of sphere bundles.
\newblock {\em Ann. of Math. (2)}, 45:294--311, 1944.

\bibitem{Swartz-higherCM}
E.~Swartz.
\newblock {$g$}-elements, finite buildings and higher {C}ohen-{M}acaulay
  connectivity.
\newblock {\em J. Combin. Theory Ser. A}, 113(7):1305--1320, 2006.

\bibitem{Swartz-finiteness}
E.~Swartz.
\newblock Topological finiteness for edge-vertex enumeration.
\newblock {\em Adv. Math.}, 219(5):1722--1728, 2008.

\bibitem{Swartz}
E.~Swartz.
\newblock Face enumeration---from spheres to manifolds.
\newblock {\em J. Eur. Math. Soc. (JEMS)}, 11(3):449--485, 2009.

\bibitem{Tay}
T.-S. Tay.
\newblock Lower-bound theorems for pseudomanifolds.
\newblock {\em Discrete Comput. Geom.}, 13(2):203--216, 1995.

\bibitem{Walkup}
D.~Walkup.
\newblock The lower bound conjecture for {$3$}- and {$4$}-manifolds.
\newblock {\em Acta Math.}, 125:75--107, 1970.

\bibitem{Whiteley}
W.~Whiteley.
\newblock Some matroids from discrete applied geometry.
\newblock In {\em Matroid theory ({S}eattle, {WA}, 1995)}, volume 197 of {\em
  Contemp. Math.}, pages 171--311. Amer. Math. Soc., Providence, RI, 1996.

\bibitem{Zheng}
H.~Zheng.
\newblock Minimal balanced triangulations of sphere bundles over the circle.
\newblock Preprint., 2015.

\end{thebibliography}
\bibliographystyle{plain}
}
\end{document}